\documentclass{amsart}

\usepackage{amsthm}
\usepackage{amssymb}
\usepackage{amsmath}
\usepackage{amsmath,amscd}
\usepackage{epsfig}

\newtheorem{theorem}{Theorem}
\newtheorem{corollary}{Corollary}
\newtheorem{lemma}{Lemma}
\newtheorem{proposition}{Proposition}

\newcommand{\RR}{{\mathbb R}}

\begin{document}

\title{Immersed self-shrinkers}

\author{Gregory Drugan}
\address{Department of Mathematics, University of Washington, Seattle, WA 98195}
\email{drugan@math.washington.edu}
\thanks{G. Drugan was partially supported by NSF RTG 0838212.}

\author{Stephen J. Kleene}
\address{Department of Mathematics, MIT, Cambridge, MA 02139}
\email{skleene@math.mit.edu}
\thanks{S.J. Kleene was partially supported by NSF DMS 1004646.}

% General info

% \date{June 2013}

\begin{abstract}
We construct infinitely many complete, immersed self-shrinkers with rotational symmetry for each of the following topological types: the sphere, the plane, the cylinder, and the torus.
\end{abstract}

\maketitle

%%%%%%%%%%%%%%%%%%%%
%%%%%%%%%%%%%%%%%%%%

\section{Introduction}

In this paper, we construct infinitely many complete, immersed self-shrinker spheres, planes, cylinders, and tori in $\mathbb{R}^{n+1}$, $n \geq 2$. A self-shrinker is an immersion $F$ from an $n$-dimensional manifold $M$ into $\RR^{n+1}$ that satisfies
\begin{equation}
\label{ss}
\Delta_g F = - \frac{1}{2}F^{\perp},
\end{equation}
where $g$ is the metric on $M$ induced by the immersion, $\Delta_g$ is the Laplace-Beltrami operator, and $F^{\perp}(p)$ is the projection of $F(p)$ into the normal space $N_p M$. The mean curvature of $F(M)$ is given by $\Delta_g F$, and when $F$ is a self-shrinker, the family of submanifolds $$M_t = \sqrt{-t}F(M)$$ is a solution to the mean curvature flow for $t \in (-\infty, 0 )$. It is a consequence of Huisken's monotonicity formula~\cite{Hu} that a solution to the mean curvature flow behaves asymptotically like a self-shrinker at a type I singularity. In addition, self-shrinkers are minimal surfaces for the conformal metric $e^{-|x|^2/(2n)}(dx_1^2 + \dots + dx_{n+1}^2)$ on $\mathbb{R}^{n+1}$.

Examples of self-shrinkers in $\RR^{n+1}$ include the sphere of radius $\sqrt{2n}$ centered at the origin, the plane through the origin, the cylinder with an axis through the origin and radius $\sqrt{2(n-1)}$, and an embedded torus $(S^1 \times S^{n-1})$ constructed by Angenent~\cite{A}. In this paper, we construct an infinite number of complete, immersed self-shrinkers.
\begin{theorem}
\label{thm:main}
There are infinitely many complete, immersed self-shrinkers in $\mathbb{R}^{n+1}$, $n \geq 2$, for each of the following topological types: the sphere $(S^n)$, the plane $(\mathbb{R}^n)$, the cylinder $(\mathbb{R} \times S^{n-1})$, and the torus $(S^1 \times S^{n-1})$.
\end{theorem}

Numerical evidence for the existence of an immersed sphere self-shrinker was provided by Angenent~\cite{A} in 1989. In 1994, Chopp~\cite{Ch} described an algorithm for constructing surfaces that are approximately self-shrinkers and provided numerical evidence for the existence of a number of self-shrinkers, including compact, embedded self-shrinkers of genus 5 and 7. Recently, Kapouleas, the second author, and M{\o}ller~\cite{KKM} and Nguyen~\cite{N1}--\cite{N3} used desingularization constructions to produce examples of complete, non-compact, embedded self-shrinkers with high genus in $\RR^3$. M{\o}ller~\cite{M} also used desingularization techniques to construct compact, embedded, high genus self-shrinkers in $\RR^3$. In~\cite{D}, the first author constructed an immersed sphere self-shrinker.

In contrast to these constructions are several rigidity theorems for self-shrinkers. Huisken~\cite{Hu} showed that the sphere of radius $\sqrt{2n}$ is the only compact, mean-convex self-shrinker in $\RR^{n+1}$, $n\geq 2$.  In their study of generic singularities of the mean curvature flow, Colding and Minicozzi~\cite{CM} showed that the only $F$-stable\footnote{Self-shrinkers are unstable as minimal surfaces for the conformal metric $e^{-|x|^2/(2n)}(dx_1^2 + \dots + dx_{n+1}^2)$ on $\mathbb{R}^{n+1}$, which can be seen by translating a self-shrinker in space (or time). To account for these translations when considering the stability of self-shrinkers, Colding and Minicozzi introduced the notion of $F$-stability (see~\cite{CM}, p.763).} self-shrinkers with polynomial volume growth in $\RR^{n+1}$, $n \geq 2$, are the sphere of radius $\sqrt{2n}$ and the plane. Ecker and Huisken~\cite{EH} showed that an entire self-shrinker graph with polynomial volume growth must be a plane in their study of the mean curvature flow of entire graphs. Afterwards, Lu Wang~\cite{W} showed that an entire self-shrinker graph has polynomial volume growth. In their classification of complete, embedded self-shrinkers with rotational symmetry, the second author and M{\o}ller~\cite{KM} showed that the sphere of radius $\sqrt{2n}$, the plane, and the cylinder of radius $\sqrt{2(n - 1)}$ are the only embedded, rotationally symmetric self-shrinkers of their respective topological type.

The self-shrinkers we construct in this paper have rotational symmetry, and they correspond to geodesics for a conformal metric on the upper-half plane: geodesics whose ends either intersect the axis of rotation perpendicularly or exit through infinity, and closed geodesics with no ends (see Figure~\ref{fig:sphere:intro} and Appendix C). The heuristic idea of the construction is to first study the behavior of geodesics near two known self-shrinkers and then use continuity arguments to find self-shrinkers between them. In order to implement this heuristic, we first give a detailed description of the basic shape and limiting properties of the geodesics: We prove that the Euclidean curvature of a non-degenerate geodesic segment, written as a graph over the axis of rotation, can vanish at no more than two points, and we also show the different ways in which a family of geodesic segments can converge to a geodesic that exits the upper-half plane. Then, after establishing the asymptotic behavior of geodesics near the plane, the cylinder, and Angenent's torus, we use induction arguments to construct infinitely many self-shrinkers near each of these self-shrinkers. A new feature of the construction is the use of the Gauss-Bonnet formula to control the shapes of geodesics that almost exit the upper-half plane. 

We note that in the one-dimensional case, the self-shrinking solutions to the curve shortening flow have been completely classified (see Gage and Hamilton~\cite{GH}, Grayson~\cite{Gr}, Abresch and Langer~\cite{AL}, Epstein and Weinstein~\cite{EW}, and Halldorsson~\cite{Hal}). One difficulty in higher dimensions ($n \geq 2$) is the presence of the $(n-1)/r$ term in the geodesic equation~(\ref{SSEq}), which allows the Euclidean curvature of a geodesic to change sign and forces a geodesic intersecting the axis of rotation $\{r=0\}$ to do so perpendicularly.

We also note that the existence of immersed $S^2$ self-shrinkers shows that the uniqueness results for constant mean curvature spheres in $\RR^3$ (see Hopf~\cite{Ho}) and for minimal spheres in $S^3$ (see Almgren~\cite{Alm}) do not hold for self-shrinkers. In addition, Alexandrov's moving plane method does not seem to have a direct application to the self-shrinker equation. (Recall that Angenent's construction of an embedded $S^1 \times S^{n-1}$ self-shrinker shows that there are compact, embedded self-shrinkers different from the sphere.) It is unknown whether or not the sphere of radius $\sqrt{2n}$ is the only embedded $S^n$ self-shrinker; however, as mentioned above, this is the only embedded $S^n$ self-shrinker with rotational symmetry.

\begin{figure}
\label{fig:sphere:intro}
\begin{center}
\includegraphics[width=.85 \textwidth]{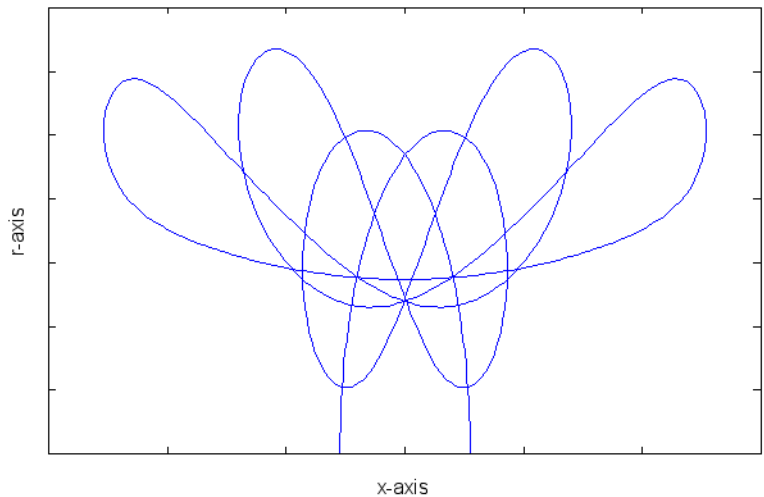}
\end{center}
\caption{A geodesic whose rotation about the $x$-axis is an immersed sphere self-shrinker.}
\end{figure}

%%%%%%%%%%%%%%%%%%%%
%%%%%%%%%%%%%%%%%%%%

\section{Preliminaries}
\label{prelim}

The self-shrinkers we construct have rotational symmetry about a line through the origin in $\mathbb{R}^{n+1}$, $n \geq 2$, and can be described by a curve in the upper half of the $(x,r)$-plane. An arclength parametrized curve $\Gamma(s) = (x(s), r(s))$ is the profile curve of a self-shrinker if and only if the angle $\alpha(s)$ solves
\begin{equation} 
\label{SSEq}
\dot{\alpha}(s) = \frac{x(s)}{2} \sin \alpha(s) + \left( \frac{n - 1}{r(s)} - \frac{r(s)}{2} \right) \cos \alpha(s),
\end{equation}
where $\dot{x}(s) = \cos \alpha(s)$ and $\dot{r}(s) = \sin \alpha(s)$. Equation~(\ref{SSEq}) is the geodesic equation for the conformal metric $g_{Ang}=r^{2(n-1)}e^{-(x^2+r^2)/2}(dx^2 + dr^2)$ on $\mathbb{H} = \{ (x,r) : x \in \mathbb{R}, \, r >0 \}$ (see~\cite{A}, pp.7-9). For $(x_0, r_0) \in \mathbb{H}$ and $\alpha_0 \in \mathbb{R}$, we let $\Gamma[x_0, r_0, \alpha_0]$ denote the unique solution to~(\ref{SSEq}) satisfying
\[
\Gamma[x_0, r_0, \alpha_0] (0) = (x_0, r_0), \quad \dot{\Gamma}[x_0, r_0, \alpha_0] (0) = (\cos(\alpha_0), \sin(\alpha_0)),
\]
and we define $\underline{\Gamma}$ to be the space of all curves $\Gamma[x_0, r_0, \alpha_0]$.

There are several particular curves of interest belonging to $\underline{\Gamma}$, namely the embedded ones.  The known embedded curves are the semi-circle $\sqrt{2n}(\cos(s), \sin (s))$, the lines $(0, s)$ and $(s, \sqrt{2(n - 1)})$, and a closed convex curve discovered by Angenent in~\cite{A}. We will refer to these curves as the sphere, the plane, the cylinder, and Angenent's torus (since the rotations of these curves about the $x$-axis respectively generate a sphere $S^n$, a plane $\mathbb{R}^n$, a cylinder $\mathbb{R} \times S^{n-1}$, and a torus $S^1 \times S^{n-1}$). For convenience, we denote the sphere, the plane, and the cylinder curves by $\mathcal{S}$, $\mathcal{P}$, and $\mathcal{C}$, respectively. It follows from a theorem of the second author and M{\o}ller in~\cite{KM} that the sphere of radius $\sqrt{2n}$, the plane, and the cylinder of radius $\sqrt{2(n - 1)}$ are the only embedded, rotationally symmetric self-shrinkers of their respective topological type. It is unknown if Angenent's torus is the only embedded, rotationally symmetric $S^1 \times S^{n-1}$ self-shrinker.

Though the metric $g_{Ang}$ and~(\ref{SSEq}) are degenerate at the boundary $\{ r = 0\}$ (the $x$-axis), there is still a smooth one parameter family of initial value problems (see the Appendix of~\cite{D}, or Theorem 2.2 in~\cite{BG}), which we denote by $Q[x_0]$, satisfying
\[
Q [x_0] (0)  = (x_0, 0), \quad \dot{Q}[x_0] (0) = (0, 1).
\]
The degeneracy of~(\ref{SSEq}) reflects the imposed axial symmetry of our surfaces, and amounts to the fact that the tangent space of  a smooth axially symmetric surface at the axis of symmetry is a perpendicular plane. We note that $Q[\sqrt{2n}]$ and $Q[0]$ are the sphere and the plane, respectively. It was shown by the first author in~\cite{D} that there is $0<x_1<\sqrt{2n}$ so that $Q[x_1]$ is the profile curve of an immersed sphere self-shrinker.

It will be useful to view a curve $\Gamma \in \underline{\Gamma}$ from three different perspectives: as a function $(x,u(x))$ over the $x$-axis, as a function $(f(r),r)$ over the $r$-axis, and as a geodesic for the metric $g_{Ang}$. The differential equations satisfied by $u(x)$ and $f(r)$ place limitations on the oscillatory behavior of $\Gamma$, and we will use these equations to describe the basic shape of the curves in $\underline{\Gamma}$. In addition, we will use the continuity properties of geodesics and the Gauss-Bonnet formula to establish convergence properties for the curves in $\underline{\Gamma}$. 

When $\Gamma \in \underline{\Gamma}$ is given as $(x, u(x))$, the function $u(x)$ satisfies the differential equation
\begin{equation}
\label{x_graph_SSEq}
\frac{u''}{1 + (u')^2} = \frac{xu'}{2} - \frac{u}{2} + \frac{n-1}{u}.
\end{equation}
This equation can be derived either directly from~(\ref{ss}) or by using the geodesic equation~(\ref{SSEq}). Differentiating~(\ref{x_graph_SSEq}), we have
\begin{equation}
\label{x_graph_SSEq_differentiated}
\frac{u'''}{1 + (u')^2} = \frac{2 u' (u'')^2}{(1 + (u')^2)^2} + \frac{x u''}{2} - \frac{n-1}{u^2}u'.
\end{equation}
Similarly, when $\Gamma$ is given as $(f(r), r)$, we have
\begin{equation}
\label{r_graph_SSEq}
\frac{f''}{1 + (f')^2} =  \left( \frac{r}{2} - \frac{n - 1}{r} \right)f' - \frac{f}{2}
\end{equation}
and
\begin{equation}
\label{r_graph_SSEq_differentiated}
\frac{f'''}{1 + (f')^2} = \frac{2 f' (f'')^2}{(1 + (f')^2)^2} + \left( \frac{r}{2} - \frac{n-1}{r} \right) f'' + \frac{n-1}{r^2}f' .
\end{equation}
We note that applying the Gauss-Bonnet formula (see~\cite{DC}, p.274) to a simple, compact region $R$ in $\mathbb{H}$ whose boundary is the piecewise smooth union of geodesic segments with external angles $\theta_0, \dots, \theta_k$, gives the formula
\begin{equation}
\label{gauss_bonnet}
\int_{R} \left( 1+ \frac{n-1}{r^2}\right) dxdr = 2\pi - \sum_{i=0}^k \theta_i.
\end{equation}

This preliminary section is divided into three parts. First, we introduce some terminology and recall some known results about self-shrinkers. Then, we study the shape of solutions to~(\ref{x_graph_SSEq}). Finally, we use the Gauss-Bonnet formula~(\ref{gauss_bonnet}) to prove some convergence results for geodesics.

%%%%%%%%%%%%%%%%%%%%

\subsection{Definitions and background}
\label{prelim:defn}

Our construction of immersed self-shrinkers follows from the study of the geometry of geodesic segments that are maximally extended as graphs over the $x$-axis.  The plane and the cylinder are degenerate in the sense that their Euclidean curvature vanishes. We will refer to a geodesic whose Euclidean curvature is not identically $0$ as \emph{non-degenerate}. We denote by $\underline{\Lambda}$ the space of non-degenerate geodesic segments that are maximally extended as graphs over the $x$-axis: $\Lambda \in \underline{\Lambda}$ if and only if $\Lambda \neq \mathcal{P}, \mathcal{C}$ is the graph of a maximally extended solution to~(\ref{x_graph_SSEq}). We note that the plane and the cylinder are the only geodesics that are not the union of elements of $\underline{\Lambda}$.

First, we describe the decomposition of a non-degenerate geodesic into the union of elements of $\underline{\Lambda}$. Given a non-degenerate geodesic of the form $\Gamma[x_0, r_0, \alpha_0]$, where $\cos(\alpha_0) \neq 0$, there exists a unique maximally extended solution $u:(a,b) \to \mathbb{R}$ to~(\ref{x_graph_SSEq}) with $u(x_0) = r_0$ and $u'(x_0) = \tan(\alpha_0)$. We define $\Lambda[0](\Gamma[x_0, r_0, \alpha_0]) \in \underline{\Lambda}$ to be the graph of $u$. If $b < \infty$ and $u(b)>0$ (see Lemma~\ref{generic_graph_behavior2}), then the geodesic $\Gamma[x_0, r_0, \alpha_0]$ can be continued past the point $(b,u(b))$, and we denote this next maximally extended geodesic segment by $\Lambda[1](\Gamma[x_0, r_0, \alpha_0])$. In general, when it is defined, we use $\Lambda[k](\Gamma[x_0, r_0, \alpha_0]) \in \underline{\Lambda}$, $k \in \mathbb{Z}$, to denote the $k^{th}$ maximally extended geodesic segment encountered in the parametrization of $\Gamma[x_0, r_0, \alpha_0]$, so that we get the (possibly finite) decomposition
\[
\Gamma[x_0, r_0, \alpha_0] = \cdots \cup \Lambda[-1](\Gamma[x_0, r_0, \alpha_0]) \cup  \Lambda[0](\Gamma[x_0, r_0, \alpha_0]) \cup  \Lambda[1] (\Gamma[x_0, r_0, \alpha_0]) \cup \cdots.
\]
When $\cos(\alpha_0) = 0$, we define $\Lambda[k](\Gamma[x_0, r_0, \alpha_0])$ similarly.

Next, we introduce a topology on $\underline{\Lambda}$. Since every $\Lambda \in \underline{\Lambda}$ intersects the $r$-axis exactly once (see Proposition~\ref{generic_graph_behavior1}), there exists a unique pair $(r_\Lambda, \alpha_\Lambda) \in \mathbb{R}^+ \times (- \pi/2, \pi/ 2)$ such that
\[
\Lambda =  \Lambda[0](\Gamma[0, r_\Lambda, \alpha_\Lambda]). 
\]
Then $\underline{\Lambda}$ carries a topology induced by the natural distance function $d$ defined by
\[
d(\Lambda_1, \Lambda_2) = |r_{\Lambda_2} - r_{\Lambda_1}| + |\alpha_{\Lambda_2} - \alpha_{\Lambda_1}|.
\]
By the continuity of geodesics, we know that a sequence $\Lambda_i \in \underline{\Lambda}$ converges smoothly to $\Lambda_\infty \in \underline{\Lambda}$ on compact subsets of $\mathbb{H}$ if and only if $d(\Lambda_i, \Lambda_\infty) \to 0$.

To give a detailed description of the shape of a geodesic, we need to identify the points where its Euclidean curvature vanishes. For a $C^2$ curve $\gamma$ in the upper half plane, we define the \emph{degree} of $\gamma$ to be the cardinality of the set where its Euclidian curvature vanishes, and we denote it by $deg(\gamma)$. In Section~\ref{prelim:x_graph} we show that each geodesic segment $\Lambda \in \underline{\Lambda}$ satisfies $deg(\Lambda) \leq 2$ (see Proposition~\ref{graphical_geodesic_degree_bound}). We denote the space of all degree $k$ curves in $\underline{\Lambda}$ by $\underline{\Lambda} (k)$, so that we have the following decomposition of $\underline{\Lambda}$:
\[
\underline{\Lambda} = \bigcup_{k = 0}^2\underline{\Lambda} (k) .
\]
Writing a geodesic segment $\Lambda \in \underline{\Lambda}(k)$ as the graph of a maximally extended solution $u:(a,b) \to \RR$ to~(\ref{x_graph_SSEq}), we know that $u''$ has a fixed sign near $b$ (since $u''$ vanishes at $k$ points). Therefore, we can decompose $\underline{\Lambda} (k)$ into the subsets $\underline{\Lambda}(k, +)$ and $\underline{\Lambda}(k, -)$, depending on the sign of $u''$ near its right end point. That is, we define $\underline{\Lambda}(k, +)$ to be the subset of $\underline{\Lambda}(k)$ consisting of maximally extended geodesic segments that are concave up near their right end points, and we define $\underline{\Lambda}(k, -)$ similarly.

In Section~\ref{prelim:gauss_bonnet} we show that the boundaries of the sets $\underline{\Lambda} (k)$ in the (non-complete) topology on $\underline{\Lambda}$ consist of curves which exit the upper-half plane either through the $x$-axis or through infinity (see Proposition~\ref{boundaries_by_type}). We refer to the elements of $\underline{\Lambda}$ that exit the upper-half plane either through the $x$-axis or through infinity as \emph{half-entire graphs}, and we denote the set of all half-entire graphs by $\underline{H}$. The geodesics $Q[x_0]$ defined above correspond to a family of half-entire graphs that exit through the $x$-axis, namely the geodesic segments $\Lambda[0](Q[x_0])$. Using the linearization of~(\ref{x_graph_SSEq}) near the sphere, Huisken's theorem on mean-convex self-shrinkers, and a comparison result for solutions to~(\ref{r_graph_SSEq}), we can prove the following result.
\begin{proposition}
\label{prop:q_sphere_intersect}
Let $Q = \Lambda[0](Q[x_0])$. Then $r_Q > \sqrt{2n}$ and $\alpha_Q<0$ when $0<x_0<\sqrt{2n}$, and $r_Q < \sqrt{2n}$ and $\alpha_Q>0$ when $x_0>\sqrt{2n}$.
\end{proposition}
\begin{proof}
The proof of this proposition follows from the results in Appendix A and Appendix B. Using the linearization of~(\ref{x_graph_SSEq}) near the sphere and Huisken's theorem, we have $\alpha_Q<0$ when $0<x_0<\sqrt{2n}$, and $\alpha_Q>0$ when $x_0>\sqrt{2n}$ (see Proposition~\ref{polar:prop1} in Appendix B). Using the comparison results from Appendix A (see Proposition~\ref{cp:prop1} and Proposition~\ref{cp:prop2}), we know that $Q$ intersects the sphere exactly once in the first quadrant, so that $r_Q > \sqrt{2n}$ when $0<x_0<\sqrt{2n}$, and $r_Q < \sqrt{2n}$ when $x_0>\sqrt{2n}$.
\end{proof}

A second family of half-entire graphs was constructed by the second author and M{\o}ller (see Theorem 3 in~\cite{KM}). They showed that for each fixed ray through the origin $r_\sigma(x) = \sigma x$, $\sigma > 0$, there exists a unique (non-entire) solution $u_\sigma$ to~(\ref{x_graph_SSEq}), called a trumpet, asymptotic to $r_\sigma$ so that: $u_\sigma$ is defined on $[0, \infty)$; $u_\sigma(0) < \sqrt{2(n-1)}$; and $u_\sigma > r_\sigma$, $0< u_\sigma'< \sigma$, and $u_\sigma'' > 0$ on $[0,\infty)$. In addition, they showed that any solution to~(\ref{x_graph_SSEq}) defined on an interval $(a, \infty)$ must be either be a trumpet $u_\sigma$  or the cylinder $u \equiv \sqrt{2(n-1)}$. An immediate consequence of this last result is that the cylinder is the only entire solution to~(\ref{x_graph_SSEq}).

Now, we introduce some notation for the previously discussed half-entire graphs:
\begin{itemize}
\item[] \emph{Inner-quarter spheres:} The set $\underline{I}^+$  of inner-quarter spheres in the first quadrant is the collection of curves of the form $ I_x  : = \Lambda[0](Q[x])$, for $0< x < \sqrt{2n}$. Each $I \in \underline{I}^+$ intersects the $r$-axis above the sphere with negative slope: 
\[
r_I > \sqrt{2n }, \quad \alpha _I  < 0.
\]

\item[] \emph{Outer-quarter spheres:} The set $\underline{O}^+$ of outer-quarter spheres in the first quadrant is the collection of curves of the form $O_x : = \Lambda[0] (Q[x])$, for $x> \sqrt{2n}$. Each $O \in \underline{O}^+$ intersects the $r$-axis below the sphere with positive slope:
\[
r_O <\sqrt{2n }, \quad \alpha_O > 0.
\]

\item[] \emph{Trumpets:} The set $\underline{T}^+$ of trumpets in the first quadrant is the collection of the graphs of $u_\sigma$, where $u_\sigma$ are the trumpets from~\cite{KM}. Each $T \in \underline{T}^+$ intersects the $r$-axis below the cylinder with a positive slope:
\[
r_T < \sqrt{2(n - 1)}, \quad \alpha_T > 0.
\]
\end{itemize}
The sets of half-entire graphs in the second quadrant: $\underline{I}^-$, $\underline{O}^-$, and $\underline{T}^-$ are defined similarly.

We also introduce the sets: $$\underline{I} = \underline{I}^+ \cup \underline{I}^- \cup \{\mathcal{S}\}, \qquad \underline{O} = \underline{O}^+ \cup \underline{O}^- \cup \{ \mathcal{S} \}, \qquad \underline{T} = \underline{T}^+ \cup \underline{T}^- .$$ In Proposition~\ref{half_entire_graph_types}, we show that the space $\underline{H}$ of half-entire graphs is the union of the sets $\underline{I}$, $\underline{O}$, and $\underline{T}$.

%%%%%%%%%%%%%%%%%%%%

\subsection{The shape of graphical geodesics}
\label{prelim:x_graph}

In this section, we study the shape of solutions to~(\ref{x_graph_SSEq}). This involves proving several results that place limitations on the possible behavior of these solutions. The main results in this section are Proposition~\ref{graphical_geodesic_degree_bound}, which shows that the Euclidean curvature $u''/(1+(u')^2)^{3/2}$ of a solution to~(\ref{x_graph_SSEq}) vanishes at no more than two points, and Proposition~\ref{half_entire_graph_types}, which addresses the classification and the shapes of half-entire graphs.

Let $u$ be a solution to~(\ref{x_graph_SSEq}). If $u$ has a local maximum (minimum) at a point $x_0$, then $u(x_0) \geq \sqrt{2(n-1)}$ ($\leq \sqrt{2(n-1)}$) with equality if and only if $u \equiv \sqrt{2(n-1)}$ is the cylinder. Also, if both $u'$ and $u''$ vanish at the same point, then $u$ must be the cylinder. Using~(\ref{x_graph_SSEq_differentiated}), when $u$ is non-degenerate\footnote{A solution to~(\ref{x_graph_SSEq}) is non-degenerate if it is not the cylinder.}, we see that $u'$ and $u'''$ have opposite signs at points where $u''=0$, so that the zeros of $u''$ are separated by zeros of $u'$.

It follows from the previous discussion that a non-degenerate solution to~(\ref{x_graph_SSEq}) has a sinsusoidal shape that oscillates between maxima above the cylinder and minima below the cylinder. In the next part of this section, we show that a maximally extended non-degenerate solution to~(\ref{x_graph_SSEq}) must intersect the $r$-axis, and its Euclidean curvature can only vanish at a finite number of points.

\begin{proposition}
\label{generic_graph_behavior1}
Let $u: (a, b) \rightarrow \mathbb{R}$ be a non-degenerate maximally extended solution to~\emph{(\ref{x_graph_SSEq})}. Then $a < 0 < b$. Moreover, $u''$ can only vanish at a finite number of points.
\end{proposition}
\begin{proof}
First, suppose $b < \infty$. We claim that $u$ cannot oscillate too much near $b$. To see this, suppose to the contrary that $u''$ vanishes in every neighborhood of $b$. Then there exists an increasing sequence $x_k \to b$ that alternates between maxima and minima of $u$. Applying the continuity of the differential equation~(\ref{x_graph_SSEq}) to the cylinder solution, we see that there is $\varepsilon >0$ so that $|u(x_k) - \sqrt{2(n-1)}|> \varepsilon$; otherwise, we can extend $u$ past $b$. Then the graph of $u(x)$ contains geodesic segments (defined as graphs over the $r$-axis on the fixed neighborhood $|r - \sqrt{2(n-1)}|<\varepsilon$) that converge to the curve $\Gamma[b, \sqrt{2(n-1)}, \pi/2]$. When $b \neq 0$, this forces the graph of $u(x)$ to become non-graphical (near $b$), and when $b=0$ this forces $u$ to extend past $b$ (see Lemma~\ref{r_graph_converge}). Thus, we have shown there is a neighborhood of $b$ in which $u''$ does not vanish.

Next, we show that $b>0$. We know that $u''$ does not vanish in a neighborhood of $b$. Examining equation~(\ref{x_graph_SSEq}) and using Lemma~\ref{x_graph_concave_below}, we see that $u'(x)$ and $u''(x)$ must have the same sign when $x$ is near $b$. If $\lim_{x \to b}u(x) \in (0,\infty)$, then $\lim_{x \to b}|u'(x)| = \infty$ (since $u$ is maximally extended), and it follows from~(\ref{x_graph_SSEq}) that $b \geq 0$. In fact, $b>0$, since the graph of $u$ is not the plane. If $\lim_{x \to b}u(x)$ is $0$ (or $\infty$), then $u'$ and $u''$ are both negative (or both positive), and the term $\frac{n-1}{u} - \frac{u}{2}$ has the correct sign to force $b>0$.

Finally, when $b = \infty$, so that $u$ is a solution to~(\ref{x_graph_SSEq}) on $(a, \infty)$, we know that $u$ is either a trumpet or the cylinder. Since $u$ is non-degenerate, it is a trumpet, and $u''>0$ on $[0,\infty)$. We conclude that $b>0$ and $u''$ does not vanish in a neighborhood of $b$. Similar arguments may be applied to the left end point $a$ to complete the proof of the lemma.
\end{proof}

The following two lemmas were used in the proof of Proposition~\ref{generic_graph_behavior1}.
\begin{lemma}
\label{r_graph_converge}
Let $f_k(r)$ be a sequence of maximally extended solutions to~\emph{(\ref{r_graph_SSEq})} defined on the neighborhood $|r - \sqrt{2(n-1)}| \leq \varepsilon$, for some $\varepsilon >0$. Suppose $f_k(\sqrt{2(n-1)})$ is an increasing sequence that converges to $b < \infty$, and $f_k'(\sqrt{2(n-1)}) \to 0$. If $b \neq 0$, then the graph of $f_k$ cannot be written as a graph over the $x$-axis for $k$ sufficiently large. If $b=0$, then $f_k$ must vanish at some point for $k$ sufficiently large.
\end{lemma}

\begin{lemma}
\label{x_graph_concave_below}
Let $u$ be a solution to~\emph{(\ref{x_graph_SSEq})} defined on a finite interval $(x_1,x_2)$. If $u'<0$ and $u''>0$ on $(x_1,x_2)$, then $\lim_{x \to x_2}u(x)>0$.
\end{lemma}

Before we prove Lemma~\ref{r_graph_converge} and Lemma~\ref{x_graph_concave_below}, we prove some properties of solutions to~(\ref{x_graph_SSEq}) and~(\ref{r_graph_SSEq}).

\begin{lemma}
\label{r_graph:crossing:above}
There exists $M_1 > \sqrt{2(n-1)}$ with the following property: Let $f$ be a solution to~\emph{(\ref{r_graph_SSEq})} with $f(\sqrt{2(n-1)}) > 0$. Suppose $f'(r) \leq 0$ when $r \geq \sqrt{2(n-1)}$. Then $f(r) < 0$ whenever $r > M_1$ and $f(r)$ is defined.
\end{lemma}
\begin{proof}
Notice that $f''(r) < 0$ when $r \geq \sqrt{2(n-1)}$, $f(r)>0$, and $f'(r) \leq 0$. Also, $f'''(r) \leq 0$  when $r \geq \sqrt{2(n-1)}$, $f'(r) \leq 0$, and $f''(r) < 0$. The idea of the proof is to use this concave down behavior to force $f$ to be negative when $r$ is large enough. Choose $r > \sqrt{2(n-1)}$ so that $f>0$ on $[\sqrt{2(n-1)},r]$. Then $$f''(r) \leq f''(\sqrt{2(n-1)}) \leq \frac{f''(\sqrt{2(n-1)})}{1+ f'(\sqrt{2(n-1)})^2} = -\frac{1}{2}f(\sqrt{2(n-1)}),$$ where we have used $f''' \leq 0$ on $[\sqrt{2(n-1)},r]$, $f''(\sqrt{2(n-1)})<0$, and equation~(\ref{x_graph_SSEq}). Integrating twice from $\sqrt{2(n-1)}$ to $r$, we have $$f(r) \leq f(\sqrt{2(n-1)}) \left[ 1 - \frac{1}{4}(r - \sqrt{2(n-1)})^2 \right].$$ Choose $M_1 = 2 + \sqrt{2(n-1)}$. Then $f(r) > 0$ whenever $r > M_1$ and $f(r)$ is defined.
\end{proof}

Next, we prove a lemma about solutions to~(\ref{r_graph_SSEq}), which shows that a positive, increasing, concave down solution cannot be defined on an interval of the form $(m, \sqrt{2(n-1)}]$, for arbitrarily small $m$.
\begin{lemma}
\label{r_graph:crossing:below}
There exists $m_1>0$ with the following property: Let $f$ be a solution to~\emph{(\ref{r_graph_SSEq})} with $f(\sqrt{2(n-1)}) > 0$. Suppose $f'(r) > 0$ and $f''(r) < 0$ when $r<\sqrt{2(n-1)}$. Then $f(r) < 0$ whenever $r < m_1$ and $f(r)$ is defined.
\end{lemma}
\begin{proof}
The idea of the proof is to use the $f' / r$ term to force $f$ to be negative when $r$ is small. We break the proof up into two steps.

Step 1: Estimate $f'$ in terms of $f$ at some point less than $\sqrt{2(n-1)}$. Without loss of generality, we assume that $f(1)$ is defined and positive. Using equation~(\ref{r_graph_SSEq_differentiated}), we see that $f'''(r) > 0$ when $r< \sqrt{2(n-1)}$ (since $f'(r) > 0$ and $f''(r) < 0$). Then, for $r < \sqrt{2(n-1)}$, we see that $f''(r) \leq f''(\sqrt{2(n-1)})$. Using equation~(\ref{x_graph_SSEq}) and the positivity of $f'$, we  have $f''(\sqrt{2(n-1)}) \leq -\frac{1}{2}f(\sqrt{2(n-1)}) \leq -\frac{1}{2}f(1)$. Therefore, $f''(r) \leq -\frac{1}{2}f(1)$. Integrating from $1$ to $\sqrt{2(n-1)}$, we arrive at the estimate $$f'(1) \geq \frac{\sqrt{2(n-1)} -1}{2}f(1).$$

Step 2: Estimate $f(r)$ for $r<1$. Suppose $f>0$ on $[r,1]$. Then using $f'>0$, $f''<0$, and~(\ref{x_graph_SSEq}), we have $$\frac{f''(r)}{f'(r)} \leq -\frac{n-1}{r}.$$ Integrating from $r$ to $1$, $$f'(r) \geq \frac{f'(1)}{r^{n-1}} \geq c_n \frac{f(1)}{r^{n-1}},$$ and integrating again $$f(r) \leq f(1) \left[ 1 - c_n \int_r^1 \frac{1}{t^{n-1}} dt \right].$$ Choose $m_1$ so that $\int_{m_1}^1 \frac{1}{t^{n-1}} dt \geq 1 / c_n$. Then $f(r) < 0$ whenever $r < m_1$ and $f(r)$ is defined.
\end{proof}

Now we prove Lemma~\ref{r_graph_converge} and Lemma~\ref{x_graph_concave_below}.
\begin{proof}[Proof of Lemma~\ref{r_graph_converge}]
Let $f_k(r)$ be a sequence of solutions to~(\ref{r_graph_SSEq}) defined on the neighborhood $|r - \sqrt{2(n-1)}| \leq \varepsilon$. Suppose $f_k(\sqrt{2(n-1)})$ is an increasing sequence that converges to $b < \infty$, and $f_k'(\sqrt{2(n-1)}) \to 0$. Let $f(r)$ denote the solution to~(\ref{r_graph_SSEq}) with $f(\sqrt{2(n-1)}) = b$ and $f'(\sqrt{2(n-1)})=0$. Then $f_k$ converges smoothly to $f$, and we note that $f''(\sqrt{2(n-1)})=-b/2$.

If $b\neq 0$, then $f''(\sqrt{2(n-1)}) \neq 0$, and for sufficiently large $k$, we have $f_k''(\sqrt{2(n-1)}) \neq 0$ so that $f_k$ cannot be written as a graph over the $x$-axis. If $b=0$, then $f \equiv 0$, and for sufficiently large $k$, we may assume that the domain of $f_k$ contains the interval $[m_1,M_1]$. Now, depending on the sign of $f_k'(\sqrt{2(n-1)})$, either $f_k'(r) \geq 0$ on $r \geq \sqrt{2(n-1)}$, or $f_k'(r) < 0$ and $f_k''(r)<0$ on $r \leq \sqrt{2(n-1)}$. Applying Lemma~\ref{r_graph:crossing:above} and Lemma~\ref{r_graph:crossing:below}, we conclude that $f_k$ crosses the $r$-axis for $k$ sufficiently large.
\end{proof}

\begin{proof}[Proof of Lemma~\ref{x_graph_concave_below}]
It is sufficient to show that a solution $f(r)$ to~(\ref{r_graph_SSEq}) defined on $(r_1,r_2)$ with $f \leq M$, $f' <0$, and $f''>0$ satisfies $r_1>0$. Let $\alpha(r) = \pi/2 - \arctan f'(r)$ so that $$\frac{d}{dr}( \log \cos \alpha(r) ) = \frac{r}{2} - \frac{n-1}{r} - \frac{f(r)}{2f'(r)} \leq \frac{r}{2} - \frac{n-1}{r} + \frac{M}{2(-f'(r_2))}.$$ Integrating from $r_1$ to $r_2$, $$\log \left( \frac{\cos \alpha(r_2)}{\cos \alpha(r_1)} \right) \leq \frac{(r_2)^2}{4} + (n-1) \log \left( \frac{r_1}{r_2}\right) + \frac{M r_2}{2(-f'(r_2))}.$$ Therefore, $$r_1 \geq r_2 \left[ -\cos \alpha(r_2) e^{-\frac{(r_2)^2}{4} + \frac{M r_2}{2f'(r_2)}} \right]^{1/(n-1)}.$$
\end{proof}

We make note of a result used during the proof of Proposition~\ref{generic_graph_behavior1}.
\begin{lemma}
\label{generic_graph_behavior1.5}
Let $u: (a, b) \rightarrow \mathbb{R}$ be a maximally extended non-degenerate solution to~\emph{(\ref{x_graph_SSEq})}. Then $u'$ and $u''$ do not vanish in a neighborhood of $b$ \emph{(or $a$)}. Moreover, $u'$ and $u''$ have the same sign \emph{(}have different signs\emph{)} in this neighborhood.
\end{lemma}
\begin{proof}
The lemma is true when $b=\infty$ by Theorem 3 in~\cite{KM}. We assume that $b<\infty$. We also assume, from the proof of Proposition~\ref{generic_graph_behavior1}, that $u''$ does not vanish in a neighborhood of $b$. Using Lemma~\ref{x_graph_concave_below}, we know that $u'$ and $u''$ must have the same sign when $u$ exits through the $x$-axis at $b$. If $u$ exits through infinity (which does not happen when $b<\infty$), then given the previously described sinusoidal shape of $u$, we must have $u'(x), u''(x) \to \infty$ as $x \to b$. Finally, when $0 < \lim_{x\to b} u(x) < \infty$, it follows that $\lim_{x\to b}|u'(x)| =\infty$, and $u'$ and $u''$ have the same sign near $b$.
\end{proof}

Now that we've finished the proof of Proposition~\ref{generic_graph_behavior1}, we want to study the oscillatory behavior of solutions to~(\ref{x_graph_SSEq}). We begin by showing that a maximally extended solution $u:(a,b) \to \mathbb{R}$ to~(\ref{x_graph_SSEq}) cannot exit through infinity when $b$ is finite.
\begin{lemma}
\label{generic_graph_behavior2}
Let $u: (a, b) \rightarrow \mathbb{R}$ be a maximally extended solution to~\emph{(\ref{x_graph_SSEq})}. If $b$ \emph{(or $a$)} is finite, then $\lim_{x \to b} u(x) < \infty$ \emph{(or $\lim_{x \to a} u(x) < \infty$)}.
\end{lemma}
\begin{proof}
We know from Lemma~\ref{generic_graph_behavior1.5} that the $u'(x)$ and $u''(x)$ have the same sign (and do not vanish) as $x$ approaches $b$. When $u'$ and $u''$ are both negative near $b$, so that $u$ is decreasing, the lemma holds. When $u'$ and $u''$ are both positive near $b$, we will show that $\lim_{x\to b }u(x) < \infty$. To see this, suppose to the contrary that $\lim_{x\to b} u(x) = \infty$. Then (by the sinusoidal shape of $u$) there is a point $x_1>0$ for which $u(x_1) = x_1 u'(x_1)$. We consider the function $\Psi(x) = xu' -u$ (from Lemma 1 in~\cite{KM}). If $x > 0$, then $\Psi' = xu'' > \frac{1}{2}x(1+(u')^2)\Psi$, and hence $\Psi(x)>0$ when $x>x_1$. Therefore, $u'(x)>0$ and $u''(x)>0$ for $x>x_1$. We note that $u''/(1+(u')^2) \geq (n-1)/u$ when $x>x_1$.

We will use a third derivative argument to show $u(b) < \infty$. Let $\psi = u'$. By the previous discussion, we have $\psi > 0$ and $\psi' > 0$ on $(x_1,b)$, and $\psi(x_1) = \frac{u(x_1)}{x_1}$. Using equation~(\ref{x_graph_SSEq_differentiated}), for $x>x_1$, we have $$\psi'' \geq \frac{1}{2}x_1 \psi' \psi^2,$$ where we also used $u''/(1+(u')^2) \geq (n-1)/u$ when $x>x_1$.

Now, for small $\varepsilon > 0$, consider the function $$\phi_{\varepsilon}(x) = \frac{M}{\sqrt{b - \varepsilon -x}}.$$ We choose $M>0$ so that $\frac{u(x_1)}{x_1} \leq \frac{M}{\sqrt{b - x_1}}$ and $\frac{3}{M^2} \leq x_1$. Then $$\phi_{\varepsilon}'' \leq \frac{1}{2} x_1 \phi_{\varepsilon}' \phi_{\varepsilon}^2,$$ and $\phi_\varepsilon(x_1) > \psi(x_1)$. Suppose $\phi_\varepsilon -\psi$ is negative at some point in $(x_1, b - \varepsilon)$. Since $\phi_\varepsilon(x_1) > \psi(x_1)$ and $\phi_\varepsilon(b - \varepsilon) = \infty$, we know that $\phi_\varepsilon -\psi$ achieves a negative minimum at some point $x_0 \in (x_1, b - \varepsilon)$. Computing $(\phi_\varepsilon -\psi)''$ at $x_0$, we arrive at a contradiction: $$0 \leq (\phi_\varepsilon -\psi)''(x_0) \leq \frac{1}{2} x_1 \phi_{\varepsilon}' (\phi_{\varepsilon}^2 - \psi^2)(x_0) < 0.$$ Therefore, $\psi \leq \phi_{\varepsilon}$. Taking $\varepsilon \to 0$ and integrating we see that $u$ is bounded from above at $b$.
\end{proof}

Next, we show that a solution to~(\ref{x_graph_SSEq}) must be convex when it perpendicularly intersects the $r$-axis below the cylinder.

\begin{lemma}
\label{shooting_horizontal_below_cylinder_is_degree_zero}
Let $u: [0, b) \to \mathbb{R}$ be a solution to~\emph{(\ref{x_graph_SSEq})} satisfying $$u(0) < \sqrt{2(n-1)}, \quad u'(0) = 0.$$ Then $u(x)$ is strictly convex on $[0, b_t)$.
\end{lemma}
\begin{proof}
Let $u_t:(a_t,b_t) \to \mathbb{R}$ be the maximally extended solution to~(\ref{x_graph_SSEq}) satisfying $u_t(0)=t$, $u_t'(0)=0$. When $t<\sqrt{2(n-1)}$, we have $u''(0) >0$, and if $u_t$ is not strictly convex, then $u_t$ has a sinusoidal shape and obtains a local maximum at a first point $y_t >0$. In particular, there is a first point $z_t > 0$ such that $$u_t (z_t) = \sqrt{2(n - 1)}.$$ Examining equation (\ref{x_graph_SSEq}), we conclude that $u_t$ is a strictly convex function on $[0, z_t]$. Applying Lemma~\ref{r_graph:crossing:below} to $u_t$ written as a graph over the $r$-axis, we see that $z_t$ cannot exist when $t < m_1$, and therefore $u_t$ is strictly convex for small $t>0$.

We will use continuity to show that $u_t$ is strictly convex for all $0<t< \sqrt{2(n-1)}$. Let $t_0>0$ be the first initial height for which $u_{t_0}$ is not strictly convex. By continuity, we know that $u_{t_0}$ is convex (otherwise, $u_t$ would not be convex for some  $t<t_0$), and thus $u_{t_0}$ does not have a sinusoidal shape with a local maximum at some first point. Therefore, we must have $t_0 \geq \sqrt{2(n-1)}$ (and hence $t_0 = \sqrt{2(n-1)}$), which completes the proof of the lemma.
\end{proof}

Using the continuity of geodesics, we can prove a more general version of the previous lemma. By considering the family of shooting problems: $u_t(t) = r_0$, $u_t'(t) = 0$, for $t \in [ 0, x_0]$, where $r_0 < \sqrt{2(n-1)}$ and $x_0>0$, we can show that the solution $u_t:[t,b_t) \to \mathbb{R}$ to~(\ref{x_graph_SSEq}) is strictly convex on $[t,b_t)$. For the continuity argument to work we assume $u_t$ is maximally extended at $b_t$ and use the facts that $b_t < \infty$ and $\lim_{x \to b_t} u_t(x) < \infty$. We have the following result.

\begin{lemma}
\label{monotonicity_after_minimum}
Let $u: [x_0, b) \to \mathbb{R}$ be a solution to~\emph{(\ref{x_graph_SSEq})} with 
\[
u(x_0) < \sqrt{2(n - 1)}, \quad u'(x_0) = 0,
\]
where $x_0 > 0$. Then $u(x)$ is strictly convex on $[x_0, b)$.
\end{lemma}

The next lemma shows that solutions which intersect the $r$-axis below the cylinder with negative slope are convex in the first quadrant.
\begin{lemma}
\label{convexity_of_sub_horizontal_shooting}
Let $u: [0, b) \to \mathbb{R}$ be a solution to~\emph{(\ref{x_graph_SSEq})} satisfying
\[
u(0) < \sqrt{2(n - 1)}, \quad u'(0) < 0.
\]
Then $u$ is strictly convex on $[0, b)$.
\end{lemma}
\begin{proof}
We assume that $u$ is maximally extended at $b$. Since $u(0) < \sqrt{2(n - 1)}$ and $u'(0) < 0$, we know that $u$ remains strictly convex until it reaches its first minimum, and we also know that $b<\infty$. It then follows (from Lemma~\ref{generic_graph_behavior1.5}) that $u$ has a minimum somewhere on $[0, b)$. Appealing to Lemma \ref{monotonicity_after_minimum} applied to the first minimum on $[0,b)$ proves the lemma.
\end{proof}

Slightly adapting the proof of Lemma~\ref{convexity_of_sub_horizontal_shooting} we can show that solutions to~(\ref{x_graph_SSEq}) which intersect the $r$-axis between the cylinder and the sphere with negative slope are degree 1 curves in the first quadrant.
\begin{lemma}
\label{almost_convexity_of_sub_horizontal_shooting}
Let $u: [0, b) \to \mathbb{R}$ be a solution to~\emph{(\ref{x_graph_SSEq})}, maximally extended at $b$, satisfying
\[
\sqrt{2(n-1)} \leq u(0) \leq \sqrt{2n}, \quad u'(0) < 0.
\]
Then there is a point $x_0 \in [0,b)$ so that $u''(x) \leq 0$ for $x \in [0,x_0]$, and $u''(x)>0$ for $x \in (x_0,b)$. Furthermore, there is a point $x_1 > x_0$ for which $u'(x_1) =0$.
\end{lemma}
\begin{proof}
We consider the following family of shooting problems: For $t \in (0, u(0)]$, let $u_t:[0,b_t) \to \mathbb{R}$ be the solution to~(\ref{x_graph_SSEq}) with $u_t(0) = t$ and $u_t'(0) = u'(0)$, for $t \in (0, u(0)]$. We assume that $u_t$ is maximally extended at $b_t$, and we will use the facts that $b_t < \infty$ and $\lim_{x \to b_t} u_t(x) < \infty$, which follow from Theorem 3 in~\cite{KM} and Lemma~\ref{generic_graph_behavior2}.

When $t<\sqrt{2(n-1)}$, we know (from the proof of Lemma~\ref{convexity_of_sub_horizontal_shooting}) that $u_t$ has a unique local minimum at a point $x_1^t \in (0, b_t)$ (the uniqueness follows from Lemma~\ref{monotonicity_after_minimum}). We claim that this property is true for all $t \leq u(0)$. Suppose to the contrary that $u_{t_*}$ does not have a local minimum in $(0,b_{t_*})$ for some first $t_* \leq u(0)$. Then, as $t$ increases to $t_*$, the points $p_t=(x_1^t, u_t(x_1^t))$ must exit the first quadrant of the upper-half plane. Since $u_t'(0)=u'(0)<0$ and $u_t(0) \leq u(0)$, we know that the points $p_t$ are bounded away from the $r$-axis. By continuity, since $b_{t_*} < \infty$, we know that $x_1^t$ is bounded when $t$ is less than and close to $t_*$. We also know that $u_t(x_1^t) \leq u(0)$ when $t<t_*$, and it follows that the points $p_t$ cannot exit the first quadrant through infinity. Finally, since $u_{t_*}(0) \leq \sqrt{2n}$ and $u_{t_*}'(0)<0$ we know from Proposition~\ref{prop:q_sphere_intersect} that the graph of $u_{t_*}$ is not a quarter sphere and by continuity the points $p_t$ are bounded away from the $x$-axis. Therefore, the points $p_t$ cannot exit the first quadrant, which is a contradiction. We conclude that $u_t$ has a unique local minimum at a point $x_1^t \in (0, b_t)$ for all $t \leq u(0)$.

Now, we can describe the behavior of $u_t$ when $\sqrt{2(n-1)} \leq t \leq u(0)$. We know that $u_t$ has a local minimum at $x_1^t$, and applying Lemma~\ref{monotonicity_after_minimum} we have $u_t''(x) > 0$ for $x \geq x_1^t$. Since $u_t'(0) <0$, it follows from the sinusoidal shape of $u_t$ that $u_t' =0$ at exactly one point. Using $u_t''(0) \leq 0$ and $u_t''(x_1^t) > 0$, we see that $u_t''(x_0^t)=0$ at some first point $x_0^t \in [0,x_1^t)$. Again appealing to the sinusoidal shape of $u_t$, we conclude that $u_t''(x) > 0$ for $x \in (x_0^t, x_1^t)$, which finishes the proof of the lemma.
\end{proof}

Now, we can show that a solution to~(\ref{x_graph_SSEq}) does not oscillate too much. We state the result in terms of the degree.

\begin{proposition}
\label{graphical_geodesic_degree_bound}
Let $\Lambda \in \underline{\Lambda}$ be a maximally extended geodesic segment. Then 
\[
deg(\Lambda ) \leq 2.
\]
Moreover, the only maximally extended geodesic segments with degree $2$ are type $(2,+)$.
\end{proposition}
\begin{proof}
Let $\Lambda \in \underline{\Lambda}$ be a maximally extended geodesic segment. Given the sinusoidal shape of $\Lambda$, we know that $\Lambda$ alternates between maxima and minima, and its Euclidean curvature vanishes extactly once between any successive maximum and minimum. Now, it follows from Lemma~\ref{shooting_horizontal_below_cylinder_is_degree_zero} and Lemma~\ref{monotonicity_after_minimum} that $\Lambda$ remains convex after a minimum in the first quadrant (including the $r$-axis), and a similar statement holds in the second quadrant. In particualr, $\Lambda$ can have at most two minima (one in each quadrant). Also, $\Lambda$ can have at most one maximum; otherwise $\Lambda$ would oscillate `after' a minimum, which cannot occur. It follows that $deg(\Lambda) \leq 2$. Moreover, $deg(\Lambda) = 2$ if and only if $\Lambda$ has two minima, in which case $\Lambda$ is type $(2,+)$.
\end{proof}

The next propostion shows that the quarter spheres and trumpets account for all the half-entire graphs, and it classifies them into their different types.

\begin{proposition} 
\label{half_entire_graph_types}
The space $\underline{H}$ of half-entire graphs is the union of the sets $\underline{I}$, $\underline{O}$, and $\underline{T}$. Moreover, the elements of $\underline{I}^+$ are type $(0,-)$, the elements of $\underline{O}^+$ are type $(1,-)$, and the elements of $\underline{T}^+$ are type $(0,+)$.
\end{proposition}
\begin{proof}
We know that an element of $\underline{\Lambda}$ that exits the upper-half plane through infinity must be a trumpet. We also know that a half-entire graph that exits through the $x$-axis must do so perpendicularly (use equation~(\ref{x_graph_SSEq}) and Lemma~\ref{generic_graph_behavior1.5}). Therefore, the space $\underline{H}$ of half-entire graphs is the union of the sets $\underline{I}$, $\underline{O}$, and $\underline{T}$.

Now we address the types of the half-entire graphs in the first quadrant. Let $Q[x_0]$ denote the geodesic satisfying $Q [x_0] (0)  = (x_0, 0)$ and $\dot{Q}[x_0] (0) = (0, 1)$. It was shown in Proposition 4.12 of~\cite{D} that $\Lambda[0](Q [x_0])$ is type $(0,-)$ for small $x_0$. Arguing by continuity, we see that $\Lambda[0](Q [x_0])$ is type $(0,-)$ for $0< x_0< \sqrt{2n}$. Therefore, the curves in $\underline{I}^+$ are type $(0,-)$. When $x_0 > \sqrt{2n}$, we know that $\Lambda[0](Q[x_0])$ intersects the $r$-axis below the sphere with positive slope, and it follows from Lemma~\ref{convexity_of_sub_horizontal_shooting} and Lemma~\ref{almost_convexity_of_sub_horizontal_shooting} that $\Lambda[0](Q[x_0])$ is type $(1,-)$. Therefore, the curves in $\underline{O}^+$ are type $(1,-)$. Finally, it follows from, Lemma~\ref{monotonicity_after_minimum} that the curves in $\underline{T}^+$ are type $(0, +)$ since $r_T < \sqrt{2(n - 1)}$ and $\alpha_T > 0$ for $T \in \underline{T}^+$.
\end{proof}

%%%%%%%%%%%%%%%%%%%%

\subsection{Applications of the Gauss-Bonnet formula}
\label{prelim:gauss_bonnet}

In this section we use the Gauss-Bonnet formula~(\ref{gauss_bonnet}) to prove some convergence results for geodesics. Applying the Gauss-Bonnet formula to a region with a piecewise geodesic boundary shows that the region cannot enclose a `large' area. In addition, if the region is near the $x$-axis, then it must enclose a `small' area. Using this heuristic, we show that a family of geodesics converging to a half-entire graph will converge to the half-entire graph as it leaves and returns to the upper-half plane (see Proposition~\ref{double_convergence_to_boundary}). We also use the Gauss-Bonnet formula to describe the boundaries of the sets $\underline{\Lambda}(k,\pm)$.

We begin by showing that a geodesic cannot interpolate between two different half-entire graphs in the first quadrant.
\begin{lemma}
\label{lemma:double_converge:interpolate}
Let $\Gamma_i \in \underline{\Gamma}$ be a sequence of geodesics with at least $2$ graphical components. Let $u_i = \Lambda[0](\Gamma_i)$ and $v_i = \Lambda[1](\Gamma_i)$, and suppose the sequences $u_i$ and $v_i$ converge to the half-entire graphs $u_\infty$ and $v_\infty$. Then $u_\infty = v_\infty$. The conclusion also holds when $u_\infty$ is the cylinder.
\end{lemma}
\begin{proof}
First, suppose $u_\infty$ and $v_\infty$ are both quarter spheres. Let $p$ and $q$ denote the right end points of $u_\infty$ and $v_\infty$, respectively. If $u_\infty \neq v_\infty$, then $p \neq q$, and there exists $\delta >0$ so that $|p-q| > 2 \delta$. Then for small $\varepsilon >0$, we claim there exists a rectangle $R$ of the form: $x_0 \leq x \leq x_0 + \delta$, $\varepsilon/2 \leq r \leq \varepsilon$ so that, for large $i$, the rectangle $R$ is contained in a simple region bounded by $\Gamma_i$ and the $r$-axis. To see this, we assume without loss of generality that the $x$-coordinate of $p$ is less than the $x$-coordinate of $q$. Given the sinusoidal shapes of $u_i$ and $v_i$, we know from Lemma~\ref{monotonicity_after_minimum} that $\Lambda[0](\Gamma_i)$ is type $(k_0,+)$ and hence $\Lambda[1](\Gamma_i)$ is type $(k_1,-)$, for some $k_0$ and $k_1$. Using the continuity of the differential equation~(\ref{x_graph_SSEq}), we see that $\Gamma_i$ follows along $u_\infty$ (getting arbitrarily close to $p$), then follows the $x$-axis (getting arbitrarily close to $q$), and then travels back to the $r$-axis along $v_\infty$. This proves the claim. To conclude the proof of the lemma in this case, we observe that $\int_R r^{-2}dxdr = \delta / \varepsilon$, and an application of the Gauss-Bonnet formula~(\ref{gauss_bonnet}) shows that this is impossible when $\varepsilon$ is small.

Second, suppose $u_\infty$ and $v_\infty$ are both trumpets. Then there exist rays $r_\sigma(x) = \sigma x$ and $r_\tau(x) = \tau x$ so that $u_\infty$ and $v_\infty$ are asymptotic to $r_\sigma$ and $r_\tau$, respectively. If $u_\infty \neq v_\infty$, then $\sigma \neq \tau$. Now, the wedge between $r_\sigma$ and $r_\tau$ has infinite area, and the same is true for the area of the wedge outside any compact set. Arguing as in the first case and using the property that the trumpets are asymptotic to the rays, we can show there is a simple region bounded by $\Gamma_i$ (and the $r$-axis) that encloses arbitrarily large area as $i \to \infty$. An application of the Gauss-Bonnet formula shows that this is impossible. The proof is similar when one of the trumpets is the cylinder.

Finally, suppose $u_\infty$ is a quarter sphere and $v_\infty$ is a trumpet or a cylinder. It follows from the sinusoidal shape of $u_i$ and Lemma~\ref{monotonicity_after_minimum} that $\Lambda[0](\Gamma_i)$ is type $(k_0,+)$ for some $k_0$. Then, arguing as in the previous cases, we can show there is a simple region bounded by $\Gamma_i$ and the $r$-axis that encloses arbitrarily large area as $i \to \infty$, and and an application of the Gauss-Bonnet formula shows that this is impossible.
\end{proof}

Next, we prove a lemma that describes the shape of a solution $u(x)$ to~(\ref{x_graph_SSEq}) when $u(0)$ is small or large.
\begin{lemma}
\label{lemma:double_converge:bounded}
Let $m_1$ and $M_1$ be the constants defined in \emph{Lemma~\ref{r_graph:crossing:below}} and \emph{Lemma~\ref{r_graph:crossing:above}}. If $u:(a,b) \to \mathbb{R}$ is a maximally extended solution to~\emph{(\ref{x_graph_SSEq})}, then
\begin{itemize}
\item[\emph{1.}]
If $u(0) < m_1$, then the graph of $u$ is in $\underline{\Lambda}(0,+)$.
\item[\emph{2.}]
If $u(0) > M_1$, then the graph of $u$ is in $\underline{\Lambda}(0,-)$.
\end{itemize}
Moreover, $a, b \to 0$ as $u(0) \to 0$ or $u(0) \to \infty$.
\end{lemma}
\begin{proof}
First, we treat the case where $u(0) < m_1$. If $u'(0)<0$, then it follows from Lemma~\ref{convexity_of_sub_horizontal_shooting} that $u''(x)>0$ for $x\geq 0$. When $u'(0) \geq 0$, it follows from~(\ref{x_graph_SSEq}) that $u''(x)>0$ for $x \geq 0$ as long as $u < \sqrt{2(n-1)}$ on $[0,x]$. In both cases, we observe that a portion of the geodesic $(x, u(x))$ may be written as a graph over the $r$-axis: $(f(r),r)$, where $f$ is a solution of~(\ref{r_graph_SSEq}). We claim that $u < \sqrt{2(n-1)}$ on $[0,b)$. To see this, suppose to the contrary that $u(x) \geq \sqrt{2(n-1)}$ for some $x>0$. Then we may choose $f$ so that $f(r) >0$, $f'(r)>0$, and $f''(r)<0$ when $m_1 \leq r \leq \sqrt{2(n-1)}$. Applying Lemma~\ref{r_graph:crossing:below} shows that this is impossible, and therefore $u < \sqrt{2(n-1)}$ on $[0,b)$. In particular, we have $u''>0$ on $[0,b)$.

Now, we estimate $b$ in terms of $u(0)$. If, say, $u(b) \leq 3 u(0)$, then $u(x) \leq 3 u(0)$ on $[0,b)$, and we can write equation~(\ref{x_graph_SSEq}) as $$\frac{d}{dx} \left( \arctan u' \right) =  \frac{xu' - u}{2} + \frac{n-1}{u} \geq -\frac{u(0)}{2} + \frac{n-1}{3 u(0)},$$ where we have used $xu'-u$ is increasing on $(0,b)$. Integrating from $0$ to $b$, we have $$\pi \geq \left( \frac{n-1}{3 u(0)} - \frac{u(0)}{2} \right) b,$$ and thus $b \to 0$ as $u(0) \to 0$. In general, if $u(b) = A u(0)$, where $A>1$, then
\begin{equation}
\label{lemma:double_converge:bounded:eq1}
\pi \geq \left( \frac{n-1}{A u(0)} - \frac{u(0)}{2} \right) b.
\end{equation}
Applying the Gauss-Bonnet formula to the triangle $T$ with vertices $(0,u(0))$, $(0,u(b))$, and $(b,u(b))$, we have $$4 \pi \geq \int_T \frac{n-1}{r^2}dxdr = \frac{(n-1)b}{(A-1) u(0)} \left[ \log A + \frac{1}{A} -1 \right].$$ If $A$ is sufficiently large, then
\begin{equation}
\label{lemma:double_converge:bounded:eq2}
4 \pi \geq \frac{b \log A}{u(b)} \geq \frac{b \log A}{\sqrt{2(n-1)}}.
\end{equation}
It follows that $b \to 0$ as $u(0) \to 0$: Fix $\varepsilon > 0$, and choose $u(0) < \varepsilon e^{-1/\varepsilon}$. If $A < e^{1/\varepsilon}$, then $Au(0) < \varepsilon$ and~(\ref{lemma:double_converge:bounded:eq1}) implies $b \lesssim \varepsilon$. If $A \geq e^{1/\varepsilon}$, then $\log A \geq 1 / \varepsilon$ and~(\ref{lemma:double_converge:bounded:eq2}) implies $b \lesssim \varepsilon$.

Second, we treat the case where $u(0) > M_1$. Since $u(0) > \sqrt{2(n-1)}$, we know that $u''(0) < 0$ and $b<\infty$. When $u'(0) \leq 0$, it follows from~(\ref{x_graph_SSEq}) that $u''(x) < 0$ for $x \geq 0$ as long as $u > \sqrt{2(n-1)}$ on $[0,x]$. When $u'(0) >0$, we can use the sinusoidal shape of $u$ (and Lemma~\ref{generic_graph_behavior1.5}) to conclude that the first zero of $u'$ occurs before the first zero of $u''$, and similar reasoning shows that $u''(x)<0$ as long as $u > \sqrt{2(n-1)}$ on $[0,x]$. Now, the decreasing portion of the geodesic $(x, u(x))$ can be written as a graph over the $r$-axis: $(f(r),r)$, where $f$ is a solution to~(\ref{r_graph_SSEq}) and $f(r) >0$ and $f'(r) \leq 0$ when $u(b) \leq  r \leq M_1$. Applying Lemma~\ref{r_graph:crossing:above} shows $u(b) > \sqrt{2(n-1)}$. In paticular, we have $u''<0$ on $[0,b)$.

To estimate $b$ in terms of $u(0)$, we note that the function $f$ from the above paragraph is defined on $[\sqrt{2(n-1)}, M_1]$. Using the shape of the graph of $u$ and equations~(\ref{r_graph_SSEq}) and~(\ref{r_graph_SSEq_differentiated}) we know that $f > 0$ and $f'' < 0$ on $[\sqrt{2(n-1)}, M_1]$. Let $\Gamma$ denote the geodesic corresponding to the graphs of $u$ and $f$. Then the region bounded by $\Gamma$ and the $r$-axis contains the triangle $T$ with vertices $(0,M_1)$, $(0,\sqrt{2(n-1)})$, and $(b, u(b))$. Using the Gauss-Bonnet formula, we have $4 \pi \geq \int_T dxdr = [M_1 - \sqrt{2(n-1)} ] b/2$, so that $b \to 0$ as $u(b) \to \infty.$

Finally, the same arguments apply to the left end point $a$.
\end{proof}

Next, we prove a lemma which restricts the domain of a solution to~(\ref{x_graph_SSEq}) that intersects the $r$-axis with steep negative slope.
\begin{lemma}
\label{lemma:double_converge:slope1}
Let $u$ be a maximally extended solution to~\emph{(\ref{x_graph_SSEq})} defined on the interval $(a,b)$. Then $b \to 0$ as $u'(0) \to -\infty$.
\end{lemma}
\begin{proof}
Fix $\varepsilon >0$. We will show there is $L>0$ so that $b \lesssim \varepsilon$ when $u'(0) \leq -L$. By Lemma~\ref{lemma:double_converge:bounded}, there exist positive constants $m$ and $M$ so that $b< \varepsilon$ when $u(0) < m$ or $u(0) > M$, so we may assume that $m \leq u(0) \leq M$. There are two cases to consider, depending on the shape of $u$.

Case 1: $u'(0) \leq -L$ and $u'' < 0$ on $[0,b)$. Since $u(0) \leq M$, $u'(0) \leq -L$, and $u''<0$, we know that $x \leq M/L$ whenever $u(x)$ is defined (integrate $u' \leq -L$ from $0$ to $x$). Therefore $b \leq M/L$, and we may choose $L > M / \varepsilon$.

Case 2: $u'(0) \leq -L$ and $u''(x) \geq 0$ for some $x \geq 0$. For large enough $L$ (depending only on $M$ and $\varepsilon$), using the continuity of the differential equation~(\ref{r_graph_SSEq}), we know there is a point $(c,u(c))$ so that $c<\varepsilon$ and $u(c) < \varepsilon$. By choosing $L > M / \varepsilon$, we may assume that $u''(c) > 0$ (see Case 1). Furthermore, by allowing for $c <2 \varepsilon$, we may assume $u'(c) \geq -1$. We work with $\varepsilon e^{-1/\varepsilon}$ in place of $\varepsilon$: We assume $c, u(c) < \varepsilon e^{-1/\varepsilon}$. Arguing as in the proof of Lemma~\ref{lemma:double_converge:bounded}, we write equation~(\ref{x_graph_SSEq}) as $$\frac{d}{dx} \left( \arctan u'(x) \right) =  \frac{xu'(x) - u(x)}{2} + \frac{n-1}{u(x)} \geq - \varepsilon e^{-1/\varepsilon} + \frac{n-1}{u(x)},$$ where we have used $xu(x)'-u(x)$ is increasing when $x \geq c$, along with the estimates on $c$, $u(c)$, and $u'(c)$. If $u(b) = A u(c)$, for some $A>1$, then integrating from $c$ to $b$, we have
\begin{equation}
\label{lemma:double_converge:slope:eq1}
\pi \geq \left( \frac{n-1}{A u(c)} - \varepsilon e^{-1/\varepsilon} \right) (b-c).
\end{equation}
Applying the Gauss-Bonnet formula to the triangle $T$ with vertices $(c,u(c))$, $(c,u(b))$, we have $$4 \pi \geq \int_T \frac{n-1}{r^2}dxdr = \frac{(n-1)(b-c)}{(A-1) u(c)} \left[ \log A + \frac{1}{A} -1 \right],$$ so that
\begin{equation}
\label{lemma:double_converge:slope:eq2}
4 \pi \geq \frac{(b-c) \log A}{u(b)} \geq \frac{(b-\varepsilon e^{-1/ \varepsilon}) \log A}{\sqrt{2(n-1)}},
\end{equation}
when $A$ is sufficeintly large. If $A < e^{1/\varepsilon}$, then $Au(c) < \varepsilon$ and~(\ref{lemma:double_converge:slope:eq1}) implies $b \lesssim \varepsilon$. If $A \geq e^{1/\varepsilon}$, then $\log A \geq 1 / \varepsilon$ and~(\ref{lemma:double_converge:slope:eq2}) implies $b \lesssim \varepsilon$. If $u(b) \leq u(c)$, then $$\pi \geq \left( \frac{n-1}{u(c)} - \varepsilon e^{-1/\varepsilon} \right) (b-c),$$ and we also have $b \lesssim \varepsilon$.
\end{proof}

The following result will be used in the proof of Lemma~\ref{lemma:double_converge:bounded_v} to restrict the domain of a solution to~(\ref{x_graph_SSEq}) that intersects the $r$-axis with steep positive slope.
\begin{lemma}
\label{lemma:double_converge:slope2}
Let $u$ be a maximally extended solution to~\emph{(\ref{x_graph_SSEq})} defined on the interval $(a,b)$. If $u$ has a local maximum at $x_1 >0$ and $u(x_1) > \max \{ u(0), \sqrt{2(n-1)} \} +2$, then $b< 2x_1$.
\end{lemma}
\begin{proof}
This lemma follows from the proofs of Claim 4.9 and Lemma 4.10 in~\cite{D}. Those results show that $u(b) > u(x_1)-2$, and $u(x_1 -s) \geq u(x_1 +s)$ when $s>0$. Since $u(b) > u(x_1)-2 > u(0)$, we have $b < 2x_1$. For convenience, we include proofs of these two facts.

Part 1: $u(b) > u(x_1)-2$. The graph $(x,u(x))$ for $x > x_1$ can be written as the graph $(f(r),r)$, where $f$ is a solution to~(\ref{r_graph_SSEq}). Now $f>0$ and $f'<0$ in a neighborhood of $u(x_1)$ (when $f(r)$ is defined), and using equations~(\ref{r_graph_SSEq}) and~(\ref{r_graph_SSEq_differentiated}), we also have $f''<0$ and $f'''<0$. Assuming $f'<0$, these inequalities hold when $r\geq \sqrt{2(n-1)}$. Repeatedly integrating $f'''<0$ from $r$ to $u(x_1)$, we have $0< [ 1- (u(x_1)-r)^2/4] f(r)$, so that $f'(r)=0$ for some $r>u(x_1)-2$; hence $u(b) > u(x_1)-2$.

Part 2: $u(x_1 -s) \geq u(x_1 +s)$ when $s>0$. Since $u'(x_1)=0$, using~(\ref{r_graph_SSEq}) and~(\ref{r_graph_SSEq_differentiated}), we have $$u'''(x_1) = \frac{x_1}{2} \left( \frac{n-1}{u(x_1)} - \frac{u(x_1)}{2} \right).$$ Let $\delta(s) = u(x_1+s) - u(x_1-s)$. Then $\delta(0) = \delta'(0) = \delta''(0) = 0$ and $\delta'''(0) = 2u'''(x_1) <0$. It follows that $\delta(s) < 0$ for small $s>0$. We will show that $\delta(s) < 0$  when $s>0$. Let $f$ be as in Part 1, and let $g$ be the solution to~(\ref{r_graph_SSEq}) corresponding to the graph of $u(x)$ for $x \leq x_1$. We note that there exists $0 < t < s$ so that $u(x_1+t ) = u(x_1-s)$ when $s>0$ is small. Setting $h® = f(r) +g(r)$, we have $h® = 2 x_1 + t -s  < 2 x_1$ so that $h® < 2 x_1$ when $r <u(x_1)$ is close to $u(x_1)$. We claim that $h® < 2 x_1$ for $r \in (u(b),u(x_1))$. To see this, suppose that $h® = 2 x_1$ for some $r \in (u(b),u(x_1))$. Then $h$ achieves a positive local minimum at some point $r_0 \in (u(b),u(x_1))$. At $r_0$ we have $h(r_0)>0$, $h'(r_0) = 0$, and $h''(r_0) \geq 0$, so that $$0 \leq \frac{h''(r_0)}{1+ (f'(r_0))^2} = \left(\frac{r}{2} - \frac{n-1}{r} \right) 2f'(r_0) -\frac{h(r_0)}{2}<0,$$ which is a contradiction. Therefore, $h < 2 x_1$ in $(u(b), u(x_1))$. Finally, to see $\delta(s) < 0$ when $s>0$, we suppose to the contrary that $\delta(s) =0$ for some $s>0$. Set $r= u(x_1+s) = u(x_1-s)$. Then $$2 x_1 > h® = (x_1 +s) + (x_1 - s) = 2 x_1,$$ which is a contradiction. We conclude that $\delta(s) < 0$ when $s>0$.
\end{proof}

Now, we prove an estimate for the second graphical component of a geodesic whose first graphical component is close to a half-entire graph in the first quadrant.
\begin{lemma}
\label{lemma:double_converge:bounded_v}
Let $\Gamma_i \in \underline{\Gamma}$ be a sequence of geodesic curves with at least $2$ graphical components. Let $u_i = \Lambda[0](\Gamma_i)$ and $v_i = \Lambda[1](\Gamma_i)$. Suppose the sequence $u_i$ converges to $u_\infty$, where $u_\infty$ is a half-entire graph in the first quadrant or the cylinder. Then there exist positive constansts $m$, $M$, and $L$, depending on $u_\infty$, so that $m \leq v_i(0) \leq M$ and $|v_i'(0)| \leq L$.
\end{lemma} 
\begin{proof}
By choosing $u_i$ sufficiently close to $u_\infty$, we may assume that the right end point of $u_i$ is bounded away from the $r$-axis. Applying Lemma~\ref{lemma:double_converge:bounded} and Lemma~\ref{lemma:double_converge:slope1}, we see that there are positive constants $m$, $M$, and $L$ so that $m \leq v_i(0) \leq M$ and $v_i'(0) \geq -L$. We want to find an upper bound for $v_i'(0)$.

Fix $\varepsilon > 0$, and choose $C > 4 \pi / \varepsilon$ so that $u_\infty < C$ on $[0, 2 \varepsilon]$. We may also assume that $u_i < C$ on $[0, 2 \varepsilon]$ (where we use continuity near the plane). If $v_i'(0)$ is sufficiently large, then $v_i(x_0) = 2C$ for some $x_0 < \varepsilon$. We claim that $v_i$ has a local maximum at some point in $(0,2 \varepsilon)$. Suppose to the contrary that $v_i$ has no local maximum in $(0,2 \varepsilon)$. Then the rectangle $R$: $x_0 \leq x \leq x_0 + \varepsilon$, $C \leq r \leq 2C$ is contained in a simple region bounded by the geodesic $\Gamma_i$ and the $r$-axis. Applying the Gauss-Bonnet formula we arrive at a contradiction, which proves the claim. It follows from Lemma~\ref{lemma:double_converge:slope2} that the right end point of $v_i$ is less than $4 \varepsilon$. Since the right end point of $u_i$ is bounded away from the $r$-axis, we conclude that $v_i'(0)$ has an upper bound.
\end{proof}

Combining the previous results, we have the following proposition, which deals with the convergence of geodesics to half-entire graphs.

\begin{proposition}
\label{double_convergence_to_boundary}
Let $\Gamma_i \in \underline{\Gamma}$ be a sequence of geodesics with at least $(k+2)$ graphical components, and suppose that the graphs $\Lambda[k](\Gamma_i)$ converge to a half-entire graph in the first quadrant $\Lambda_0$. Then, either $\Lambda[k + 1](\Gamma_i)$ or $\Lambda[k - 1](\Gamma_i)$ converge to $\Lambda_0$. The conclusion also holds when $\Lambda_0$ is the cylinder.
\end{proposition}
\begin{proof}
Without loss of generality, we may assume that $k$ is even. Under this assumption, the right end point of $\Lambda[k](\Gamma_i)$ is also the right end point of $\Lambda[k+1](\Gamma_i)$. To simplify notation, we let $u_i = \Lambda[k](\Gamma_i)$, $v_i = \Lambda[k+1](\Gamma_i)$, and $u_\infty = \Lambda_0$. Here we are identifying a solution to~(\ref{x_graph_SSEq}) with its graph.

With the above notation, the solutions $u_i$ converge to the half-entire graph $u_\infty$. To prove the proposition, we need to show that the sequence of initial conditions $(v_i(0),v_i'(0))$ converges to $(u_\infty(0),u_\infty'(0))$. It is sufficient to show that every subsequence of $(v_i(0),v_i'(0))$ has a subsequence converging to $(u_\infty(0),u_\infty'(0))$.

We know from Lemma~\ref{lemma:double_converge:bounded_v} that every subsequence of $(v_i(0),v_i'(0))$ has a convergent subsequence. Let $v_\infty$ be the solution of~(\ref{x_graph_SSEq}) corresponding to such a convergent subsequence. Notice that $v_\infty$ is a half-entire graph in the first quadrant (otherwise, $v_\infty$ has a right end point in the upper-half plane, and by continuity $u_i$ cannot converge to $u_\infty$). It follows from Lemma~\ref{lemma:double_converge:interpolate} that $v_\infty = u_\infty$.
\end{proof}

As an application of Proposition~\ref{double_convergence_to_boundary}, we describe the boundaries of the sets $\underline{\Lambda}(k,\pm)$ in the topology defined on $\underline{\Lambda}$. We note that the topology on $\underline{\Lambda}$ is not complete, and in particular there are sequences $\lambda_i \in \underline{\Lambda}$ that converge smoothly on compact subsets of $\mathbb{H}$ to the plane or the cylinder.

\begin{proposition} 
\label{boundaries_by_type}
The following statements hold:
\begin{displaymath}
\begin{array}{ll}
\emph{1.} \quad \partial \underline{\Lambda}(0, +) = \underline{T}, & \emph{2.} \quad  \partial \underline{\Lambda} (0, -) = \underline{I}, \\ \\
\emph{3.} \quad \partial \underline{\Lambda} (1, +) =  \underline{I}^+ \cup \underline{O}^- \cup \underline{T}^- \cup \{\mathcal{S} \} , & \emph{4.} \quad \partial \underline{\Lambda} (1, -) =  \underline{I}^- \cup \underline{O}^+ \cup  \underline{T}^+ \cup\{\mathcal{S} \} , \\  \\
\emph{5.} \quad \partial \underline{\Lambda} (2,+) = \underline{O},
\end{array}
\end{displaymath}
where $\partial$ is the boundary from the topology defined on $\underline{\Lambda}$.
\end{proposition}
\begin{proof}
It follows from the continuity of geodesics that a maximally extended geodesic graph is in the interior of some $\underline{\Lambda}(k,\pm)$ when it is not a half-entire graph. Therefore, in order to classify the boundaries of the sets $\underline{\Lambda}(k,\pm)$, it suffices to analyze the half-entire graphs.

We begin by considereing the half entire graphs in $\underline{T}$.  Let $T = \Lambda[0](\Gamma[r_T, \alpha_T])$ be a trumpet. Since the set of initial data $(r, \alpha)$ corresponding to half-entire graphs is one-dimensional, we can perturb the initial data to obtain curves $\Gamma_\epsilon : = \Gamma[r_T + \epsilon_r, \alpha_T + \epsilon_\alpha]$ so that $\Lambda[0] (\Gamma_\epsilon)$ is not in $\underline{H}$ for arbitrarily small $\epsilon$. Let $u_{i, \epsilon}: (a_{i, \epsilon}, b_{i, \epsilon} ) \rightarrow \mathbb{R}$ be the function whose graph is $\Lambda[i] (\Gamma_\epsilon)$. If $T \in \underline{T}^+$, then $a_{i, \epsilon}$ is bounded for small $\epsilon$, so that $\lim_{x \rightarrow a_i} u'_{i, \epsilon}(x) = -\infty$. By construction, we have $b_{0, 0} = \infty$ and $b_{0, \epsilon} < \infty$, for $\epsilon \neq 0$. There are two cases to consider: $(a)$ $\lim_{x \to b_{0, \epsilon}} u_{0, \epsilon}(x) = \infty$ and $(b)$ $\lim_{x \to b_{0, \epsilon} } u_{0, \epsilon}(x) = - \infty$.

In case $(a)$, we claim that $u_{0, \epsilon} (x)$ is a globally convex function. If not, then the graph of $u_{0, \epsilon} (x)$ is type $(2, +)$. Since $u_{0, \epsilon}'(0)>0$, we know that $u_{0, \epsilon}$ is convex in the second quadrant. It follows that there are points $0<x_0<x_1$ so that $u_{0, \epsilon}$ has a maximum at $x_0$ and a minimum at $x_1$. Since $u_{0, \epsilon}$ converges to a globally convex function, we have $x_0 \to \infty$ as $\epsilon \rightarrow 0$. We note that $u_{0, \epsilon}(x_1) < \sqrt{2(n-1)}$ so that $u_{0, \epsilon}$ intersects the cylinder between $x_0$ and $x_1$. Applying the Gauss-Bonnet formula to the region contained between the graph of $u_{0, \epsilon}$ and the cylinder, we arrive at a contradiction (since the area of this region approaches $\infty$ as $\epsilon \to 0$). We conclude that $u_{0, \epsilon}$ is globally convex.  Applying Proposition~\ref{double_convergence_to_boundary} to $\Gamma_\epsilon$ we see that $u_{1, \epsilon}(x) \rightarrow u_{0, 0} (x)$ as $\epsilon \to 0$, and examining the possible types of curves, we see that the graph of $u_{1, \epsilon}$ must be degree $1$ for small $\epsilon$. This says that $T $ is in the boundary of $ \underline{\Lambda}(0, +)$ and $ \underline{\Lambda} (1, -)$. In case $(b)$, we similarly conclude that the graph of  $u_{0, \epsilon}$ is type $(1, -)$ and the graph of $u_{1, \epsilon}$ is type $(0,  +)$ for small $\epsilon$. In both cases we get that  $T$ is in the boundary of $ \underline{\Lambda}(0, +)$ and $ \underline{\Lambda} (1, -)$. A similar result holds when $T \in \underline{T}^-$.

Next, we consider the half-entire graphs in $\underline{I}$. Let $I$ be an inner-quarter sphere (or the sphere). By performing a similar perturbation as above, we obtain curves $\Gamma_\epsilon$ with $\Lambda[0] (\Gamma_\epsilon) \notin \underline{H}$ converging to $I$ as $\epsilon \rightarrow 0$. If $I \in \underline{I}^+$, then it is type $(0,-)$, and an argument similar to the one in the trumpet case shows that $\Lambda[0] (\Gamma_\epsilon)$ and $\Lambda[1] (\Gamma_\epsilon)$ are type $(0,-)$ and type $(1,+)$ (or type $(1,+)$ and type $(0,-)$). It follows that $\underline{I}^+$ is contained in both $\partial \underline{\Lambda}(0, -)$ and $ \partial \underline{\Lambda}(1, +)$. A similar result holds for $\underline{I}^-$. Also, since the sphere $\mathcal{S}$ is the limit of elements in $\underline{I}^+$ (and $\underline{I}^-$), we see that $\mathcal{S}$ is in $\partial \underline{\Lambda}(0, -)$, $ \partial \underline{\Lambda}(1, -)$, and $\partial \underline{\Lambda}(1, +)$.

Lastly, we consider the outer-quarter spheres. Arguing as we did for the inner-quarter spheres, we have $\underline{O}^+$ is contained in both $\partial \underline{\Lambda}(1, -)$ and $ \partial \underline{\Lambda}(2, +)$, and a similar result holds for $\underline{O}^-$. We note that $\mathcal{S} \in \partial \underline{\Lambda}(2, +)$.

Finally, by considering the possible limiting shapes of different types of curves and using the continuity of geodesics, we can complete the proof of the proposition. For instance, the limit of type $(1,+)$ curves can only be type $(0,+)$, type $(0,-)$, or type $(1,+)$, and by continuity such a limit cannot be in $\underline{T}^+$, $\underline{I}^-$ or $\underline{O}^+$.
\end{proof}

Several convergence results follow from Proposition~\ref{boundaries_by_type}. For instance, the geodesic limit of type $(1,-)$ curves whose right end points remain bounded away from the $r$-axis in a compact subset of $\mathbb{H}$, must either be a type $(1,-)$ curve or a type $(0,-)$ curve. In particular, if these curves converge to a half-entire graph, then it must in $\underline{I}^-$. We collect some of these results in the following corollary.
\begin{corollary} 
\label{boundaries_by_type_cor}
Let $\Lambda_t$ be a family of geodesic segments in $\underline{\Lambda}(k, \pm)$ whose right end points $p_t$ remain bounded away from the $r$-axis in a compact subset of $\mathbb{H}$. If $p_t \to p_\infty$, then there exists $\Lambda_\infty \in \underline{\Lambda}$ so that $\Lambda_t \to \Lambda_\infty$ both as geodesics and in the topology defined on $\underline{\Lambda}$. Moreover, if $\Lambda_\infty$ is a half-entire graph, then the following statements hold:
\begin{quote}
\begin{tabular}{ll}
\emph{1.} If $\Lambda_t \in \underline{\Lambda}(0, +)$, then $\Lambda_\infty$ is in $\underline{T}^-$, & \emph{2.} If $\Lambda_t \in \underline{\Lambda}(0, -)$, then $\Lambda_\infty$ is in $\underline{I}^-$, \\ \\
\emph{3.} If $\Lambda_t \in \underline{\Lambda}(1, +)$, then $\Lambda_\infty$ is in $\underline{O}^-$ or $\underline{T}^-$, & \emph{4.} If $\Lambda_t \in \underline{\Lambda}(1, -)$, then $\Lambda_\infty$ is in $\underline{I}^-$, \\ \\
\emph{5.} If $\Lambda_t \in \underline{\Lambda}(2, +)$, then $\Lambda_\infty$ is in $\underline{O}^-$.
\end{tabular}
\end{quote}
\end{corollary}

%%%%%%%%%%%%%%%%%%%%
%%%%%%%%%%%%%%%%%%%%

\section{Shooting problems}
\label{shooting}

Our construction of immersed self-shrinkers involves the study of two shooting problems for the geodesic equation~(\ref{SSEq}). In one of the shooting problems, we shoot perpendicularly from the $x$-axis and study the geodesics $Q[x_0] = \Gamma[x_0,0,\pi/2]$. In the other shooting problem, we shoot perpendicularly from the $r$-axis and study the geodesics $\Gamma[0,r_0,0]$. In both cases, the goal is to find a geodesic whose $k^{th}$ component is a half-entire graph. The rotation of such a geodesic about the $x$-axis is a self-shrinker. In addition, we note that a geodesic, from one of these shooting problems, whose $k^{th}$ component intersects the $r$-axis perpendicularly also corresponds to a self-shrinker.

In Section~\ref{prelim:gauss_bonnet}, we showed that the boundaries of the sets $\underline{\Lambda}(k, \pm)$ are half entire graphs (see Proposition~\ref{boundaries_by_type} and Corollary~\ref{boundaries_by_type_cor}). It follows that whenever a continuous family of geodesic segments in $\underline{\Lambda}$ changes type, it must move through a half-entire graph. Therefore, we can construct self-shrinkers by finding solutions to the shooting problems whose components eventually have different types. In order to construct infinitely many self-shrinkers in this way, we first establish the asymptotic behavior of geodesics near the plane, the cylinder, and Angenent's torus.

%%%%%%%%%%%%%%%%%%%%

\subsection{Behavior of geodesics near the plane}
\label{shoot_near_plane}

To begin, we consider the continuous family of geodesics $Q[t] = \Gamma[t,0,\pi/2]$ obtained by shooting perpendicularly from the $x$-axis. By Proposition~\ref{half_entire_graph_types}, we know the types of the geodesic graphs $\Lambda[0](Q[t])$, and we are interested in describing the shapes of the graphs $\Lambda[k](Q[t])$ when $t>0$ is small. The following two lemmas are consequences of several results from Section~\ref{prelim}.
\begin{lemma}
\label{shooting_plane_down}
Let $\Gamma = \Gamma[0,r_0,\alpha_0]$ be a geodesic with $r_0 \in (m_1, \sqrt{2(n-1)})$ and $\alpha_0 \in (-\pi/2,0)$. Then $\Lambda[1](\Gamma)$ exists, and  for $\alpha_0$ sufficiently close to $-\pi/2$, we have $\Lambda[1](\Gamma)$ is type $(0,-)$ with $r_{\Lambda[1](\Gamma)} \in (\sqrt{2n}, M_1)$ and $\alpha_{\Lambda[1](\Gamma)} \in (-\pi/2,0)$. Moreover, $\alpha_{\Lambda[1](\Gamma)} \to -\pi/2$ as $\alpha_0 \to -\pi/2$.
\end{lemma}
\begin{proof}
Let $u:(a,b) \to \mathbb{R}$ denote the maximally extended solution to~(\ref{x_graph_SSEq}) whose graph is the geodesic segment $\Lambda[0](\Gamma)$. By assumption $u(0) <\sqrt{2(n-1)}$ and $u'(0) < 0$, and it follows from the work in Section~\ref{prelim} that $b<\infty$ and $0<u(b)<\infty$, and $u$ is convex on $[0,b)$ (see Section~\ref{prelim:defn}, Lemma~\ref{generic_graph_behavior2}, and Lemma~\ref{convexity_of_sub_horizontal_shooting}).  Since $b$ and $u(b)$ are finite, we conclude that $\Lambda[1](\Gamma)$ exists.

When $\alpha_0 \to -\pi/2$, we have $u'(0) \to -\infty$, and it follows from Lemma~\ref{lemma:double_converge:slope1} that $b \to 0$. We note that $\Lambda[0](\Gamma)$ achieves its minimum over $[0,b)$ at an interior point. By the continuity of equation~(\ref{r_graph_SSEq}), this minimum approaches $0$ as $\alpha_0 \to -\pi/2$. Applying Lemma~\ref{r_graph:crossing:below}, we have $u(b) < \sqrt{2(n-1)}$.

Now, let $f$ denote the solution to~(\ref{r_graph_SSEq}) with $f(u(b))=b$ and $f'(u(b))=0$. We note that $f$ is concave down, and again using the continuity of equation~(\ref{r_graph_SSEq}), we see that the domain of $f$ approaches $(0,\infty)$ as $\alpha_0 \to -\pi/2$. Applying Lemma~\ref{r_graph:crossing:above} we conclude that $f$ crosses the $r$-axis below $M_1$, and the slope at this point approaches $0$ as $\alpha_0 \to -\pi/2$. In addition, when $b$ is small, the comparison arguments used in the proof of Lemma~\ref{cp:lemma:1} in the Appendix show that $f$ crosses the sphere at least once in the first quadrant. Also, when $b$ is small, the slope of $f(r)$ may be chosen small for $r \in [u(b),\sqrt{2n}]$ (since $f$ is close to the plane), and we see that $f$ crosses the sphere exactly once in the first quadrant. Therefore, $r_{\Lambda[1](\Gamma)} \in (\sqrt{2n}, M_1)$ and $\alpha_{\Lambda[1](\Gamma)} \to -\pi/2$ as $\alpha_0 \to -\pi/2$.
\end{proof}

\begin{lemma}
\label{shooting_plane_up}
Let $\Gamma = \Gamma[0,r_0,\alpha_0]$ be a geodesic with $r_0 \in (\sqrt{2n}, M_1)$ and $\alpha_0 \in (0,\pi/2)$. Then $\Lambda[1](\Gamma)$ exists, and for $\alpha_0$ sufficiently close to $\pi/2$, we have $\Lambda[1](\Gamma)$ is type $(0,+)$ with $r_{\Lambda[1](\Gamma)} \in (m_1, \sqrt{2(n-1)})$ and $\alpha_{\Lambda[1](\Gamma)} \in (0,\pi/2)$. Moreover, $\alpha_{\Lambda[1](\Gamma)} \to \pi/2$ as $\alpha_0 \to \pi/2$.
\end{lemma}
\begin{proof}
Let $u:(a,b) \to \mathbb{R}$ denote the maximally extended solution to~(\ref{x_graph_SSEq}) whose graph is the geodesic segment $\Lambda[0](\Gamma)$. By assumption $u(0) > \sqrt{2n}$ and $u'(0) > 0$, and it follows from the work in Section~\ref{prelim} that $b<\infty$, $0<u(b)<\infty$ (see Section~\ref{prelim:defn} and Lemma~\ref{generic_graph_behavior2}). Since $b$ and $u(b)$ are finite, we conclude that $\Lambda[1](\Gamma)$ exists. In addition, using the sinusoidal shape of $u$, we note that $u$ achieves a local maximum at some point $x_1>0$.

When $\alpha_0 \to \pi/2$, we have $u'(0) \to \infty$, and it follows from the continuity of equation~(\ref{r_graph_SSEq}), that $u(x_1) \to \infty$. Using Part 1 in the proof of Lemma~\ref{lemma:double_converge:slope2}, we have $\Lambda[1](\Gamma)$ is type $(0,+)$ and $u(b) > u(x_1) -2$ for $\alpha_0$ sufficiently close to $\pi/2$. The triangle with vertices $(0,\sqrt{2(n-1)})$, $(0,\sqrt{2n})$, and $(x_1,u(x_1))$ is contained in a simple region bounded by $\Gamma$ and the $r$-axis, and it follows from the Gauss-Bonnet formula~(\ref{gauss_bonnet}) that $x_1 \to 0$ as $u(x_1) \to \infty$. Applying Lemma~\ref{lemma:double_converge:slope2}, we see that $b \to 0$ as $\alpha_0 \to \pi/2$.

Now, let $f$ denote the solution to~(\ref{r_graph_SSEq}) with $f(u(b))=b$ and $f'(u(b))=0$. We note that $f$ is concave down so that $f'(\sqrt{2n})>0$, and using equation~(\ref{x_graph_SSEq}) we have $f'(\sqrt{2n}) < \sqrt{n/2}f(\sqrt{2n})$. Then using $f(\sqrt{2n}) < b$ and the continuity of equation~(\ref{r_graph_SSEq}), at the point $r= \sqrt{2n}$, we see that the domain of $f$ approaches $(0,\infty)$ as $\alpha_0 \to \pi/2$. Applying Lemma~\ref{r_graph:crossing:below} we conclude that $f$ crosses the $r$-axis above $m_1$, and the slope at this point approaches $0$ as $\alpha_0 \to \pi/2$. Therefore, $r_{\Lambda[1](\Gamma)} \in (m_1, \sqrt{2n})$, and $\alpha_{\Lambda[1](\Gamma)} \to \pi/2$ as $\alpha_0 \to \pi/2$.
\end{proof}

Now, we can describe the asymptotic behavior of the geodesics $Q[t]$ near the cylinder.
\begin{proposition}
\label{asymptotic_near_plane}
For each $N>0$, there exists $\varepsilon>0$ so that whenever $0<t<\varepsilon$, the geodesic segment $\Lambda[k](Q[t])$ exists for $0 \leq k \leq N$. Moreover, $\Lambda[k](Q[t])$ is type $(0,-)$ when $k$ is even, and $\Lambda[k](Q[t])$ is type $(0,-)$ when $k$ is odd.
\end{proposition}
\begin{proof}
We know that $\Lambda[0](Q[t])$ is type $(0,-)$ when $t< \sqrt{2n}$. We also know that $r_{\Lambda[0](Q[t])} \in (\sqrt{2n}, M_1)$ and $\alpha_{\Lambda[0](Q[t])} \in (-\pi/2,0)$. By continuity, we have $\alpha_{\Lambda[0](Q[t])} \to -\pi/2$ as $t \to 0$. Then, applying Lemma~\ref{shooting_plane_up}, we see that $\Lambda[1](Q[t])$ exists, and for $t$ sufficiently close to $0$, we have $\Lambda[1](Q[t])$ is type $(0,+)$ with $r_{\Lambda[1](Q[t])} \in (m_1, \sqrt{2(n-1)})$ and $\alpha_{\Lambda[1](Q[t])} \in (-\pi/2,0)$. Moreover, $\alpha_{\Lambda[1](Q[t])} \to -\pi/2$ as $t \to 0$. Applying Lemma~\ref{shooting_plane_down} to $\Lambda[1](Q[t])$ shows that $\Lambda[2](Q[t])$ exists, and  for $t$ sufficiently close to $0$, we have $\Lambda[2](Q[t])$ is type $(0,-)$ with $r_{\Lambda[2](Q[t])} \in (\sqrt{2n}, M_1)$ and $\alpha_{\Lambda[2](Q[t])} \in (-\pi/2,0)$. Moreover, $\alpha_{\Lambda[2](Q[t])} \to -\pi/2$ as $t \to 0$. The proposition follows from repeated applications of Lemma~\ref{shooting_plane_up} and Lemma~\ref{shooting_plane_down}.
\end{proof}

%%%%%%%%%%%%%%%%%%%%

\subsection{Behavior of geodesics near the cylinder}
\label{shoot_near_cylinder}

Next, we study the continuous family of geodesics $\Gamma_t = \Gamma[0,t,0]$ obtained by shooting perpendicularly from the $r$-axis. By Lemma~\ref{shooting_horizontal_below_cylinder_is_degree_zero} we know that $\Lambda[0](\Gamma_t)$ is type $(0,+)$ when $t<\sqrt{2(n-1)}$. The following lemma about geodesics near the cylinder will be used to describe the shape of $\Gamma_t$ when $t$ is close to $\sqrt{2(n-1)}$.
\begin{lemma}
\label{shooting_cylinder_flat}
Let $\Gamma = \Gamma[0,r_0,\alpha_0]$ with $r_0 < \sqrt{2n}$ and $\alpha_0 \in (-\pi/2,0)$. Then $\Lambda[1](\Gamma)$ exists, and for $(r_0,\alpha_0)$ sufficiently close to $(\sqrt{2(n-1)},0)$, the geodesic segment $\Lambda[1](\Gamma)$ is type $(1,-)$ with $\alpha_{\Lambda[1](\Gamma)} \in (0, \pi/2)$. Moreover, $\Lambda[1](\Gamma)$ converges to the cylinder as $(r_0,\alpha_0) \to (\sqrt{2(n-1)},0)$.
\end{lemma}
\begin{proof}
By the work in Section~\ref{prelim}, we know that $\Lambda[0](\Gamma)$ is type $(k_0,+)$ and it has a finite right end point. Therefore, $\Lambda[1](\Gamma)$ exists and has type $(k_1,-)$. Applying Proposition~\ref{double_convergence_to_boundary}, we see that $\Lambda[1](\Gamma)$ converges to the cylinder as $(r_0,\alpha_0) \to (\sqrt{2(n-1)},0)$. In particular, $r_{\Lambda[1](\Gamma)} \to \sqrt{2(n-1)}$. Using Lemma~\ref{almost_convexity_of_sub_horizontal_shooting} and the continuity of geodesics, we observe that a maximally extended graphical geodesic segment, which intersects the $r$-axis perpendicularly between the sphere and the cylinder, is type $(2,+)$. Combining this observation with the work in Section~\ref{prelim:x_graph} shows that $\alpha_{\Lambda[1](\Gamma)} \in (0, \pi/2)$, and consequently $\Lambda[1](\Gamma)$ is type $(1,-)$.
\end{proof}

Now, we can describe the asymptotic behavior of the geodesics $\Gamma[0,r_0,0]$ near the cylinder.
\begin{proposition}
\label{asymptotic_near_cylinder}
Let $\Gamma_t = \Gamma[0,t,0]$. For each $N>0$, there exists $\varepsilon>0$ so that whenever $\sqrt{2(n-1)} - \varepsilon < t < \sqrt{2(n-1)}$, the geodesic segment $\Lambda[k](\Gamma_t)$ exists for $0 \leq k \leq N$. Moreover, $\Lambda[k](\Gamma_t)$ is type $(1,-)$ when $k$ is odd, and $\Lambda[k](\Gamma_t)$ is type $(1,+)$ when $k \geq 2$ is even.
\end{proposition}
\begin{proof}
By Lemma~\ref{shooting_horizontal_below_cylinder_is_degree_zero}, $\Lambda[0](\Gamma_t)$ is type $(0,+)$. Arguing as in the proof of Lemma~\ref{shooting_cylinder_flat}, we see that $\Lambda[1](\Gamma_t)$ exists, and for $t$ sufficiently close to $\sqrt{2(n-1)}$, the geodesic segment $\Lambda[1](\Gamma_t)$ is type $(1,-)$ with $\alpha_{\Lambda[1](\Gamma_t)} \in (0, \pi/2)$. Moreover, $\Lambda[1](\Gamma_t)$ converges to the cylinder as $t \to \sqrt{2(n-1)}$. The proposition follows from repeated applications of Lemma~\ref{shooting_cylinder_flat}.
\end{proof}

%%%%%%%%%%%%%%%%%%%%

\subsection{Behavior of geodesics near Angenent's torus}
\label{shoot_near_torus}

We continue the study of the geodesics $\Gamma_t = \Gamma[0,t,0]$ by illustrating two procedures for constructing self-shrinkers. We prove the result due to Angenent~\cite{A} that there is an embedded torus self-shrinker, and we also prove the result from~\cite{D} that there is an immersed sphere self-shrinker.

Consider the geodesics $\Gamma_t = \Gamma[0,t,0]$, where $t<\sqrt{2(n-1)}$. From Lemma~\ref{shooting_horizontal_below_cylinder_is_degree_zero} we know that $\Lambda[0](\Gamma_t)$ is type $(0,+)$, and from the work in Section~\ref{prelim}, we know that $\Lambda[1](\Gamma_t)$ exists. Proposition~\ref{asymptotic_near_cylinder} tells us that $\Lambda[1](\Gamma_t)$ is type $(1,-)$ when $t$ is close to $\sqrt{2(n-1)}$. Moreover, $\Lambda[1](\Gamma_t)$ has a local maximum in the first quadrant. When $t$ is close to $0$, it follows from the proof of Lemma~\ref{shooting_plane_down} that $\Lambda[1](\Gamma_t)$ is type $(0,-)$ with a local maximum in the second quadrant.

There are two notable differences in the geodesics $\Lambda[1](\Gamma_t)$ when $t$ is close to $\sqrt{2(n-1)}$ and when $t$ is close to $0$. One difference is the location of the local maximum, and the other difference is the curve type. As $t$ decreasees from $\sqrt{2(n-1)}$ to $0$, there is a first initial height $t = r_{Ang}$ for which the local maximum of $\Lambda[1](\Gamma_t)$ intersecsts the $r$-axis. (More rigorously, let $r_{Ang}$ denote the infimum of the set of $r<\sqrt{2(n-1)}$ with the property that for $t>r$, the maximum of $\Lambda[1](\Gamma_t)$ occurs in the first quadrant.) Then $r_{Ang} > 0$, and consequently $\Gamma_{r_{Ang}}$ is a closed geodesic whose rotation about the $x$-axis is an embedded torus self-shrinker. We will refer to $\Gamma_{r_{Ang}}$ as Angenent's torus.

Notice that $\Lambda[1](\Gamma_{r_{Ang}})$ is type $(0,-)$. In particualr, as $t$ decreases from $\sqrt{2(n-1)}$ to $r_{Ang}$, the geodesic segments $\Lambda[1](\Gamma_t)$ change type. Therefore, $\Lambda[1](\Gamma_t)$ must be in $\partial \underline{\Lambda}(1,-)$ for some $t < \sqrt{2(n-1)}$. By continuity, the right end points of $\Lambda[1](\Gamma_t)$ remain bounded away from the $r$-axis in a compact subset of $\mathbb{H}$ when $t$ is between, say, $\sqrt{2(n-1)} - \varepsilon$ and $r_{Ang}$, and applying Corollary~\ref{boundaries_by_type_cor} we see that there is $r_1 > r_{Ang}$ so that $\Lambda[1](\Gamma_{r_1}) \in \underline{I}^-$. The rotation of the geodesic $\Gamma_{r_1}$ about the $x$-axis is an immersed sphere self-shrinker.

We end this section with a description of the behavior of geodesics near Angenet's torus $\Gamma_{r_{Ang}}$.
\begin{proposition}
\label{asymptotic_near_torus}
Let $\Gamma_t = \Gamma[0,t,0]$. For each $N>0$, there exists $\varepsilon>0$ so that whenever $r_{Ang}< t < r_{Ang} + \varepsilon$, the geodesic segment $\Lambda[k](\Gamma_t)$ exists for $0 \leq k \leq N$. Moreover, $\Lambda[k](\Gamma_t)$ is type $(0,+)$ when $k$ is even and $\Lambda[k](\Gamma_t)$ is type $(0,-)$ when $k$ is odd.
\end{proposition}
\begin{proof}
The proposition follows from the continuity of geodesics and the convexity of $\Gamma_{r_{Ang}}$.
\end{proof}

%%%%%%%%%%%%%%%%%%%%
%%%%%%%%%%%%%%%%%%%%

\section{Construction of self-shrinkers}
\label{construct}

In this section we construct an infinite number of sphere and plane self-shrinkers near the plane (Theorem~\ref{thm:near_plane}), an infinite number of sphere and tori self-shrinkers near the cylinder (Theorem~\ref{thm:near_cylinder}), and an infinite number of sphere and cylinder self-shrinkers near Angenent's torus (Theorem~\ref{thm:near_torus}).

%%%%%%%%%%%%%%%%%%%%

\begin{theorem}
\label{thm:near_plane}
There is a decreasing sequence $t_0 > t_1 > \cdots$ so that the rotation of the geodesic $Q[t_k]$ about the $x$-axis is an $S^n$ self-shrinker when $k$ is even and a complete $\mathbb{R}^n$ self-shrinker when $k$ is odd. Moreover, $Q[t_k]$ is the union of $k+1$ maximally extended geodesic segments.
\end{theorem}
\begin{proof}
The proof is by induction. For the base case, we define $t_0=\sqrt{2n}$ so that $Q[t_0] = \mathcal{S}$ is the sphere. We note that $\Lambda[0](Q[t])$ is type $(0,-)$ for $0<t<t_0$ (this follows from Proposition~\ref{half_entire_graph_types}). We also note that $\Lambda[1](Q[t])$ exists for $0<t<t_0$, since the sphere is the only profile curve corresponding to an embedded $S^n$ self-shrinker.

Continuing the base case, we know that there is $\varepsilon >0$ so that the left end point of $\Lambda[1](Q[t])$ remains bounded away from the $r$-axis in a compact subset of $\mathbb{H}$ for $\varepsilon \leq t \leq t_0 - \varepsilon$. When $t$ is close to $0$ it follows from Proposition~\ref{asymptotic_near_plane} that $\Lambda[1](Q[t])$ is type $(0,+)$. Applying Proposition~\ref{double_convergence_to_boundary} and Proposition~\ref{boundaries_by_type} to the left end point  of $\Lambda[0](Q[t_0])$ shows that $\Lambda[1](Q[t])$ is either type $(1,-)$ or type $(2,+)$ when $t$ is close to $t_0 = \sqrt{2n}$. In particular, by choosing $\varepsilon$ small enough, we see that the geodesic segments $\Lambda[1](Q[t])$ change type as $t$ increases from $\varepsilon$ to $t_0 - \varepsilon$. It follows that $\Lambda[1](Q[t])$ is a half-entire graph for some $t$ between $\varepsilon$ and $t_0 - \varepsilon$.  We define $t_1$ to be the first $t>0$ such that $\Lambda[1](Q[t])$ is a half-entire graph. Then $t_1 < t_0$, and by Corollary~\ref{boundaries_by_type_cor} the geodesic segment $\Lambda[1](Q[t_1])$ is a trumpet in the first quadrant.

For the inductive case, we assume that $t_N < t_{N-1} < \cdots < t_0$ are defined: $t_k$ is the first $t>0$ such that $\Lambda[k](Q[t])$ is a half-entire graph. We also assume that $\Lambda[i](Q[t_k])$ is type $(0,-)$ when $i \leq k$ is even, and it is type $(0,+)$ when $i \leq k$ is odd. In addition, we assume that $\Lambda[k](Q[t_k])$ is an inner-quarter sphere in the second quadrant when $k$ is even, and it is a trumpet in the first quadrant when $k$ is odd. Now, suppose $N$ is odd.  Then $\Lambda[N+1](Q[t])$ exists for $0<t<t_N$, and there is $\varepsilon >0$ so that the right end point of $\Lambda[N+1](Q[t])$ remains bounded away from the $r$-axis in a compact subset of $\mathbb{H}$ for $\varepsilon \leq t \leq t_N - \varepsilon$. When $t$ is close to $0$ it follows from Proposition~\ref{asymptotic_near_plane} that $\Lambda[N+1](Q[t])$ is type $(0,+)$. Applying Proposition~\ref{double_convergence_to_boundary} and Proposition~\ref{boundaries_by_type} shows that $\Lambda[N+1](Q[t])$ is type $(1,-)$ when $t$ is close to $t_N$. In particular, by choosing $\varepsilon$ small enough, we may assume that the geodesic segments $\Lambda[N+1](Q[t])$ change type as $t$ increases from $\varepsilon$ to $t_N - \varepsilon$. It follows that $\Lambda[N+1](Q[t])$ is a half-entire graph for some $t$ between $\varepsilon$ and $t_N - \varepsilon$. We define $t_{N+1}$ to be the first $t>0$ such that $\Lambda[N+1](Q[t])$ is a half-entire graph. Then $t_{N+1} < t_N$, and by Corollary~\ref{boundaries_by_type_cor} the geodesic segment $\Lambda[N+1](Q[t_{N+1}])$ is an inner-quarter sphere in the second quadrant. This completes the inductive case when $N$ is odd. When $N$ is even, the argument is similar to the construction of $Q[t_1]$ from $Q[t_0]$.
\end{proof}

%%%%%%%%%%%%%%%%%%%%

\begin{theorem}
\label{thm:near_cylinder}
There is an increasing sequence $t_0 < t_1 < \cdots < \sqrt{2(n-1)}$ so that the rotation of the geodesic $\Gamma[0,t_k,0]$ about the $x$-axis is an $S^1 \times S^{n-1}$ self-shrinker when $k$ is even and an $S^n$ self-shrinker when $k$ is odd. Moreover, $\Gamma[0,t_k,0]$ is the union of $k+2$ distinct maximally extended geodesic segments.
\end{theorem}
\begin{proof}
The proof is by induction; it is similar to the proof of Theorem~\ref{thm:near_plane}. Let $\Gamma_t = \Gamma[0,t,0]$. By Lemma~\ref{shooting_horizontal_below_cylinder_is_degree_zero}, we know that $\Lambda[0](\Gamma_t)$ is type $(0,+)$ when $t<\sqrt{2(n-1)}$. It follows from Proposition~\ref{asymptotic_near_cylinder} that for each $N>0$, there exists $\varepsilon>0$ so that whenever $\sqrt{2(n-1)} - \varepsilon < t < \sqrt{2(n-1)}$, the geodesic segment $\Lambda[k](\Gamma_t)$ exists for $0 \leq k \leq N$. Moreover, $\Lambda[k](\Gamma_t)$ is type $(1,-)$ when $k$ is odd, and $\Lambda[k](\Gamma_t)$ is type $(1,+)$ when $k \geq 2$ is even.

For the base case, we define $t_0$ to be the largest $t< \sqrt{2(n-1)}$ such that $\Lambda[1](\Gamma_t)$ intersects the $r$-axis perpendicularly. Then $t_0 = r_{Ang}$ and $\Gamma_{t_1}$ is the closed embedded convex curve constructed in Section~\ref{shoot_near_torus}.

Continuing the base case, since $\Lambda[1](\Gamma_t)$ is type $(1,-)$ when $t$ is close to $\sqrt{2(n-1)}$ and type $(0,-)$ when $t$ is close to $t_0$, there is some $t$ between $\sqrt{2(n-1)}$ and $t_0$ such that $\Lambda[1](\Gamma_t)$ is a half-entire graph. We define $t_1$ to be the largest $t< \sqrt{2(n-1)}$ such that $\Lambda[1](\Gamma_t)$ is a half-entire graph. By Corollary~\ref{boundaries_by_type_cor}, the half-entire graph $\Lambda[1](\Gamma_{t_1})$ is an inner-quarter sphere in the second quadrant. Notice that $\Lambda[0](\Gamma_{t_1})$ is type $(0,+)$, $\Lambda[1](\Gamma_{t_1}) \in \underline{I}^-$, and $\Lambda[-1](\Gamma_{t_1}) \in \underline{I}^+$ so that $\Gamma_{t_1}$ is the union of $3$ maximally extended geodesic segments.

For the inductive case, we assume that $t_{2N-1} > t_{2N-3} > \cdots > t_1$ are defined: $t_{2k-1}$ is the largest $t< \sqrt{2(n-1)}$ such that $\Lambda[k](\Gamma_t)$ is a half-entire graph. We also assume that $\Lambda[N](\Gamma_{t_{2N-1}})$ is an inner-quarter sphere. Suppose $\Lambda[N](\Gamma_{t_{2N-1}})$ is an inner-quarter sphere in the first quadrant. By Proposition~\ref{double_convergence_to_boundary} and and Proposition~\ref{boundaries_by_type} we have $\Lambda[N+1](\Gamma_t)$ is type $(0,-)$ with a local maximum in the second quadrant when $t>t_{2N-1}$ is close to $t_{2N-1}$. It follows that the type of $\Lambda[N+1](\Gamma_t)$ changes as $t$ decreases from $\sqrt{2(n-1)}$ to $t_{2N-1}$, so we can define $t_{2N+1} > t_{2N-1}$ to be the largest $t< \sqrt{2(n-1)}$ such that $\Lambda[N+1](\Gamma_t)$ is a half-entire graph. Then $\Lambda[N+1](\Gamma_{t_{2N+1}})$ is an inner-quarter sphere in the second quardrant (since it is the limit of type $(1,-)$ geodesic segments). Therefore, the geodesic $\Gamma_{t_{2N+1}}$ is the union of $2N+3$ maximally extended geodesic segments, and its rotation about the $x$-axis is an immersed $S^n$ self-shrinker. Furthermore, since the local maximums of $\Lambda[N+1](\Gamma_{t_{2N+1}})$ and $\Lambda[N+1](\Gamma_{t})$ for $t$ near $t_{2N-1}$ are in different quadrants, there exists $t_{2N}$ between $t_{2N-1}$ and $t_{2N+1}$ so that $\Lambda[N+1](\Gamma_{t_{2N}})$ intersects the $r$-axis perpendicularly. Then $\Gamma_{t_{2N}}$ is the union of $2N+2$ maximally extended geodesic segments, and its rotation about the $x$-axis is an immersed $S^1 \times S^{n-1}$ self-shrinker. This completes the inductive case.
\end{proof}

%%%%%%%%%%%%%%%%%%%%

\begin{theorem}
\label{thm:near_torus}
There is a decreasing sequence $t_0 > t_1 > \cdots > r_{Ang}$ so that the rotation of the geodesic $\Gamma[0,t_k,0]$ about the $x$-axis is a complete $\mathbb{R}^1 \times S^{n-1}$ self-shrinker when $k$ is even and an $S^n$ self-shrinker when $k$ is odd. Moreover, $\Gamma[0,t_k,0]$ is the union of $2k+1$ maximally extended geodesic segments.
\end{theorem}
\begin{proof}
The proof is by induction; it is similar to the proofs of Theorem~\ref{thm:near_plane} and Theorem~\ref{thm:near_cylinder}, and we provide a sketch. Let $\Gamma_t = \Gamma[0,t,0]$. Given $N>0$, there exists $\varepsilon > 0$ so that whenever $r_{Ang} < t < r_{Ang} + \varepsilon$, the geodesic segment $\Lambda[k](\Gamma_t)$ exists for $0 \leq k \leq N$. Moreover, $\Lambda[k](\Gamma_t)$ is type $(0,+)$ when $k$ is even, and it is type $(0,-)$ when $k$ is odd. For the base case, we define $t_0 = \sqrt{2(n-1)}$, so that $\Gamma_{t_0}= \mathcal{C}$ is the cylinder. For the general case, we define $t_k$ to be the first $t>r_{ang}$ such that $\Lambda[k](\Gamma_t)$ is a half-entire graph. Then $\Lambda[k](\Gamma_{t_{k}})$ is either a trumpet in the first quadrant or an inner-quarter sphere in the second quadrant, depending on whether it is the limit of type $(0,+)$ curves or $(0,-)$ curves, respectively.
\end{proof}

%%%%%%%%%%%%%%%%%%%%
%%%%%%%%%%%%%%%%%%%%

\section*{Appendix A: Comparison results for quarter spheres}

In this appendix we prove some comparison results for quarter spheres. The main application of these results is that an inner-quarter sphere first intersects the $r$-axis outside of the sphere, and an outer-quarter sphere first intersects the $r$-axis inside the sphere.

Let $f$ and $g$ be solutions to
\begin{equation}
\label{cp:eqn:1}
\frac{f^{\prime \prime }}{1+(f^{\prime })^{2}}=\left( \frac{r}{2}-\frac{n-1}{r}\right) f^{\prime }-\frac{1}{2}f.
\end{equation}
We are interested in the shooting problem where $f^{\prime }(0)=g^{\prime}(0)=0$ and $g(0) > f(0) > 0$.

In particular, we will consider the case where $g$ is the sphere ($g(r)=\sqrt{2n-r^2}$) and $f$ is an inner-quarter sphere. In this setting, we know that $f$ is decreasing, concave down, and it crosses the $r$-axis before its slope blows-up (see Section 2 and Corollary 3.3 in~\cite{D}). We want to show that $f$ crosses the $r$-axis outside of the sphere.

We will use the following identities at $0$ for solutions to~(\ref{cp:eqn:1}):
\begin{equation}
\label{cp:eqn:2}
f'(0)=0, \quad f''(0)= -\frac{1}{2n}f(0), \quad f'''(0) =0, \quad f^{(iv)}(0) = -\frac{3}{4n(n+2)} \left(\frac{f(0)^3}{n^2}+f(0) \right).
\end{equation}
These identities follow from l'H\^{o}spital's rule applied to~(\ref{cp:eqn:1}) and its derivatives.

\begin{lemma}
\label{cp:lemma:1}
If $f$ is a solution to~\emph{(\ref{cp:eqn:1})} with $f(0)<\sqrt{2n}$, then $f$ must intersect the sphere before it crosses the $r$-axis.
\end{lemma}
\begin{proof}
Suppose $f$ does not intersect $g$ (the sphere) before it crosses the $r$-axis. Let $r_{0}>0$ be the point where $f$ crosses the $r$-axis.  Then $g>f$ on $[0,r_{0})$.

We consider the function $v=\frac{f}{g}$, which satisfies
\begin{equation*}
v^{\prime }=\frac{f^{\prime }-vg^{\prime }}{g}
\end{equation*}
and
\begin{equation*}
v^{\prime \prime }=\frac{f^{\prime \prime }-vg^{\prime \prime }}{g}-2\frac{g^{\prime }}{g}v^{\prime }.
\end{equation*}
Now, using l'H\^{o}spital's rule and the above identities at $0$, we have $v^{\prime }(0)=0$ and $v^{\prime \prime }(0)=0$. Similarly, $v'''(0)=0$, and $v^{(iv)}(0)=-\frac{3}{4n^3(n+2)} v(0) [f(0)^{2}-g(0)^{2}]>0$, where we use the non-linear dependence of $f^{(iv)}(0)$ on $f(0)$ in the last equality. It follows that $v$ is increasing near $0$. Since $v(0)>0$ and $\lim_{r\rightarrow r_{0}}v(r)=0$ (where we use l'H\^{o}spital's rule if $r_{0}=\sqrt{2n}$), we see that $v$ must achieve its supremum over $(0,r_{0})$ at an interior point, say $\bar{r}$.

By assumption, $0<v<1$ in $(0,r_{0})$, and in particular $0<v(\bar{r})<1$. We compute $v''(\bar{r})$. Using $v'(\bar{r})=0$, we have at $\bar{r}$: $$v g'' = (1+(g')^2) \left[ \left( \frac{r}{2} - \frac{n-1}{r} \right) f' - \frac{1}{2} f \right]$$ so that $$f'' - v g'' = \left( (f')^2 - (g')^2 \right) \left[ \left( \frac{r}{2} - \frac{n-1}{r} \right) f' - \frac{1}{2} f \right] = (v^2-1) (g')^2 \frac{f''}{1+(f')^2}.$$ Therefore, at $\bar{r}$: $$v'' = \frac{f^{\prime \prime }-vg^{\prime \prime }}{g} = (v^2-1) \frac{(g')^2}{g} \frac{f''}{1+(f')^2} > 0,$$ which contradicts the fact that $v$ has a maximum at $\bar{r}$.
\end{proof}

\begin{lemma}
\label{cp:lemma:2}
If $f$ is a solution to~\emph{(\ref{cp:eqn:1})} with $f(0)<\sqrt{2n}$, then $f$ can only intersect the sphere once before it crosses the $r$-axis.
\end{lemma}
\begin{proof}
Suppose $f$ intersects $g$ (the sphere) at two points before it crosses the $r$-axis: say $r_1$ and $r_2$, where $0<r_1<r_2$. Since $g(0) > f(0)$, we may assume that $f>g$ on $(r_1,r_2)$. We may also assume that $r_2 < \sqrt{2n}$ (otherwise, $f$ would intersect the $r$-axis perpendicularly at $\sqrt{2n}$, contradicting Huisken's theorem).

We consider the function $w=\frac{f^{\prime }}{g^{\prime }}$, which satisfies
\begin{equation*}
w^{\prime }=\frac{f^{\prime \prime }-wg^{\prime \prime }}{g^{\prime }}
\end{equation*}
and
\begin{equation*}
w^{\prime \prime }=\frac{f^{\prime \prime \prime }-wg^{\prime \prime \prime }}{g^{\prime }}-2\frac{g^{\prime \prime }}{g^{\prime }}w^{\prime }.
\end{equation*}
Now, using l'H\^{o}spital's rule and the above identities at $0$, we have $w(0)= \frac{f''(0)}{g''(0)} = \frac{f(0)}{g(0)}<1$ and $w^{\prime }(0)=0$. Furthermore, using the non-linear dependence of $g^{(iv)}(0)$ on $g(0)$, we have $w^{\prime \prime }(0)=\frac{1}{2n^2(n+2)}w(0)[f(0)^{2}-g(0)^{2}]<0$. It follows that $w$ is decreasing near $0$. By assumption, we have $w(r_2) \geq 1$,  and consequently, $w$ must achieve its infimum on $(0,r_{2})$ at an interior point, say $\bar{r}$.

Recall
\begin{equation*}
f^{\prime \prime \prime }=\left( 1+(f^{\prime })^{2}\right) \left[ \frac{2f^{\prime }(f^{\prime \prime })^{2}}{(1+(f^{\prime })^{2})^{2}}+\left(\frac{r}{2}-\frac{n-1}{r}\right) f^{\prime \prime }+\frac{n-1}{r^{2}}f^{\prime }\right] .
\end{equation*}
Using $w'(\bar{r})$, we have at $\bar{r}$: $$\frac{2f^{\prime }(f^{\prime \prime })^{2}}{1+(f^{\prime })^{2}}-w\frac{2g^{\prime }(g^{\prime \prime })^{2}}{1+(g^{\prime })^{2}} = 2f^{\prime }\left[ \frac{(f^{\prime \prime })^{2}}{1+(f^{\prime })^{2}}-\frac{(g^{\prime\prime })^{2}}{1+(g^{\prime })^{2}}\right] = 2f^{\prime }(g^{\prime \prime })^{2}\left[ \frac{w^{2}-1}{(1+(f^{\prime})^{2})(1+(g^{\prime })^{2})}\right].$$ Also, at $\bar{r}$:
\begin{equation*}
\left[ \left( \frac{r}{2}-\frac{n-1}{r}\right) f^{\prime \prime }+\frac{n-1}{r^{2}}f^{\prime }\right] -w\left[ \left( \frac{r}{2}-\frac{n-1}{r}\right) g^{\prime \prime }+\frac{n-1}{r^{2}}g^{\prime }\right] =0
\end{equation*}
and
\begin{eqnarray*}
&&(f^{\prime })^{2}\left[ \left( \frac{r}{2}-\frac{n-1}{r}\right) f^{\prime \prime }+\frac{n-1}{r^{2}}f^{\prime }\right] -w \, (g^{\prime })^{2}\left[\left( \frac{r}{2}-\frac{n-1}{r}\right) g^{\prime \prime }+\frac{n-1}{r^{2}}g^{\prime }\right] \\
&& \qquad = \, w \, (g^{\prime })^{2}(w^{2}-1)\left[ \left( \frac{r}{2}-\frac{n-1}{r}\right) g^{\prime \prime }+\frac{n-1}{r^{2}}g^{\prime }\right].
\end{eqnarray*}
Therfore, at $\bar{r}$:
\begin{eqnarray*}
w^{\prime \prime } &=& \frac{f^{\prime \prime \prime }-wg^{\prime \prime
\prime }}{g^{\prime }} \\
&=& \left\{ \frac{2w(g^{\prime \prime })^{2}}{(1+(f^{\prime})^{2})(1+(g^{\prime })^{2})}+wg^{\prime }\left[ \left( \frac{1}{2}r-\frac{n-1}{r}\right) g^{\prime \prime }+\frac{n-1}{r^{2}}g^{\prime }\right] \right\} (w^{2}-1) \\
&=& \left\{ \frac{2w(g^{\prime \prime })^{2}}{(1+(f^{\prime})^{2})(1+(g^{\prime })^{2})}-wg^{\prime }\frac{r}{(2n-r^{2})^{3/2}}\right\}(w^{2}-1),
\end{eqnarray*}
where we used $g(r)=\sqrt{2n-r^{2}}$ in the last equality. Since $w(\bar{r}) < w(0) < 1$, we have $w''(\bar{r}) < 0$, which contradicts the fact that $w$ has a minimum at $\bar{r}$.
\end{proof}

We have the following consequence of Lemma~\ref{cp:lemma:1} and Lemma~\ref{cp:lemma:2}.
\begin{proposition}
\label{cp:prop1}
An inner-quarter sphere intersects the sphere exactly once before it crosses the $r$-axis.
\end{proposition}

When $f(0) >\sqrt{2n}$, similar arguments show that $f$ blows-up at a point $r_* < \sqrt{2n}$. Assuming $f(r_*) > 0$, this follows from the proofs of Lemma~\ref{cp:lemma:1} and Lemma~\ref{cp:lemma:2}. A proof that $f(r_*) > 0$ when $f(0) > \sqrt{2n}$ is given in Appendix B, where we study the linearized rotational self-shrinker differential equation near the sphere. In fact, we show that an outer-quarter sphere (viewed as a graph over the $x$-axis) has a local maximum and no local minima in the first quadrant. Therefore, we have the following result.
\begin{proposition}
\label{cp:prop2}
An outer-quarter sphere intersects the sphere exactly once before it crosses the $r$-axis.
\end{proposition}

%%%%%%%%%%%%%%%%%%%%
%%%%%%%%%%%%%%%%%%%%

\section*{Appendix B: A Legendre type differential equation}

In this appendix, we study the behavior of the first graphical component of quarter spheres near the sphere. For simplicity, we will use the term `quarter sphere' and the notation $Q$ to refer to the first graphical component $\Lambda[0](Q)$ of the quarter sphere $Q$. Following the analysis in Appendix A of~\cite{KKM}, where the $n=2$ case is treated, we show that the linearization of the rotational self-shrinker differential equation near the sphere is a Legendre type differential equation. An analysis of this differential equation shows that outer-quarter spheres in the first quadrant intersect the $r$-axis with positive slope, and inner-quarter spheres in the first quadrant intersect the $r$-axis with negative slope.

Writing the rotational self-shrinker differential equation in polar coordinates $\rho = \rho(\phi)$, where $\rho = \sqrt{x^2+r^2}$ and $\phi = \arctan(r/x)$, we have
\begin{equation}
\label{polar_SSEq}
\rho'' = \frac{1}{\rho} \left\{ \rho'^2 + \left( \rho^2 + \rho'^2 \right) \left[ n - \frac{\rho^2}{2} - (n-1)\frac{\rho'}{\rho \tan{\phi} } \right] \right\}.
\end{equation}
In these coordinates, the sphere corresponds to the constant solution $\rho = \sqrt{2n}$. We note that this equation has a singularity when $\phi =0$ due to the $1 / \tan{\phi}$ term.

Making the substitution $\psi = 1 - \cos{\phi}$, we can write equation~(\ref{polar_SSEq}) as
\begin{eqnarray}
\label{change_polar_SSEq}
\psi \frac{d^2 \rho}{d \psi^2} & = & \frac{1}{\rho (2-\psi)} \left( \rho^2 + \psi (2 - \psi) \left( \frac{d \rho}{d \psi} \right)^2 \right) \left[ n - \frac{\rho^2}{2} - (n-1) \frac{ (1-\psi)}{\rho} \frac{d \rho}{d \psi} \right] \\
& & + \frac{\psi}{\rho} \left( \frac{d \rho}{d \psi} \right)^2 - (1- \psi) \frac{d \rho}{d \psi}, \nonumber
\end{eqnarray}
which has the form of the singular Cauchy problem studied in~\cite{BG} (where $\psi$ is the time variable). Applying Theorem 2.2 in~\cite{BG} to~(\ref{change_polar_SSEq}), shows that the solution $\rho(\phi,\epsilon)$ to~(\ref{polar_SSEq}) with $\rho(0,\epsilon) = \sqrt{2n} + \epsilon$ and $\frac{d \rho}{d \phi}(0,\epsilon)=0$ depends smoothly on $(\phi,\epsilon)$ in a neighborhood of $(0,0)$. It then follows from the smooth dependence on initial conditions (away from the singularity at $\phi = 0$) for solutions to~(\ref{polar_SSEq}) that $\rho(\phi,\epsilon)$ is smooth when $\phi \in [0,\pi/2]$ and $\epsilon$ is close to $0$.

In order to understand the behavior of $\rho(\phi,\epsilon)$ when $\epsilon$ is close to $0$, we study the linearization of the rotational self-shrinker differential equation near the sphere $\rho(\phi,0) = \sqrt{2n}$. We define $w$ by $$w(\phi) = \frac{d}{d \epsilon} \bigg|_{\epsilon =0} \rho(\phi,\epsilon).$$ Then $w$ satisfies the (singular) linear differential equation:
\begin{equation}
\label{linear_polar_SSEq}
w'' = -\frac{n-1}{\tan{\phi}} w' - 2n w,
\end{equation}
with $w(0) = 1$ and $w'(0) = 0$. We will show that $w(\pi/2)<0$ and $w'(\pi/2)<0$.

\begin{lemma}
\label{polar:lemma1}
Let $w$ be the solution to~\emph{(\ref{linear_polar_SSEq})} with $w(0) = 1$ and $w'(0) = 0$. Then $w(\pi/2)<0$ and $w'(\pi/2)<0$.
\end{lemma}
\begin{proof}
We begin by making the substitution $\xi = \cos{\phi}$, which turns~(\ref{linear_polar_SSEq}) into the following Legendre type differential equation:
\begin{equation}
\label{legendre:eq}
(1-\xi^2) \frac{d^2w}{d \xi^2} = n \xi \frac{d w}{d \xi} - 2nw,
\end{equation}
with the initial conditions at $\xi=1$: $$w(1)=1, \qquad \frac{d w}{d \xi}(1) = 2.$$ To prove the lemma, we need to show that $w = w(\xi)$ satisfies $w(0) <0$ and  $\frac{d w}{d \xi}(0) > 0$.

Taking derivatives of~(\ref{legendre:eq}) we have the following second order differential equations:
\begin{equation}
\label{legendre:eq:diff1}
(1-\xi^2) \frac{d^3w}{d \xi^3} = (n+2) \xi \frac{d^2 w}{d \xi^2} - n\frac{d w}{d \xi},
\end{equation}
\begin{equation}
\label{legendre:eq:diff2}
(1-\xi^2) \frac{d^4w}{d \xi^4} = (n+4) \xi \frac{d^3 w}{d \xi^3} + 2 \frac{d^2 w}{d \xi^2}.
\end{equation}
It follows from~(\ref{legendre:eq}) and~(\ref{legendre:eq:diff1}) that $$\frac{d^2 w}{d \xi^2}(1) = \frac{2n}{n+2}, \qquad \frac{d^3w}{d \xi^3}(1) = -\frac{4n}{(n+2)(n+4)}.$$ We also note that the differential equation~(\ref{legendre:eq:diff2}) for $\frac{d^2 w}{d \xi^2}$ satisfies a maximum principle.

An analysis of the possible values of $\frac{d^2 w}{d \xi^2}(0)$ and $\frac{d^3 w}{d \xi^3}(0)$ shows that $\frac{d^2 w}{d \xi^2}(0) >0$ and $\frac{d^3 w}{d \xi^3}(0)<0$ are the only conditions that agree with the initial conditions at $\xi = 1$. For example, if $\frac{d^2 w}{d \xi^2}(0) < 0$ and $\frac{d^3 w}{d \xi^3}(0) > 0$, then the conditions $\frac{d^2 w}{d \xi^2}(1) >0$ and $\frac{d^3 w}{d \xi^3}(1)<0$ imply that $\frac{d^2 w}{d \xi^2}$ achieves a positive maximum on $[0,1]$ at an interior point, which contradicts the maximum principle. 

Since $\frac{d^2 w}{d \xi^2}(0) >0$ and $\frac{d^3 w}{d \xi^3}(0)<0$, it follows from~(\ref{legendre:eq}) and~(\ref{legendre:eq:diff1}) that $w(0) <0$ and  $\frac{d w}{d \psi}(0) > 0$. Regarding $w = w(\phi)$ as a function of $\phi$, this says that $w(\pi/2)<0$ and $w'(\pi/2)<0$, which proves the lemma.
\end{proof}

Now we can prove the assertion about quarter spheres made at the beginning of this appendix.
\begin{proposition}
\label{polar:prop1}
The first graphical component of an outer-quarter sphere in the first quadrant intersects the $r$-axis with positive slope, and the first graphical component of an inner-quarter sphere in the first quadrant intersects the $r$-axis with negative slope.
\end{proposition}
\begin{proof}
It follows from Lemma~\ref{polar:lemma1} that the proposition is true for the quarter sphere $Q[x_0]$ when $x_0$ is close to $\sqrt{2n}$. In fact, when $x_0$ is close to $\sqrt{2n}$, we know that $Q[x_0]$ is $C^2$ close to the sphere in the first quadrant, and applying Lemma~\ref{polar:lemma1} we have the following description of the shape of $Q[x_0]$ in the first quadrant: If $Q[x_0]$ is an inner-quarter sphere, then it is strictly convex and monotone in the first quadrant, and if $Q[x_0]$ is an outer-quarter sphere, then it is strictly convex with a local maximum in the first quadrant.

To prove the proposition in general, we first consider the case of inner-quarter spheres. Suppose to the contrary that some inner-quarter sphere intersects the $r$-axis with non-negative slope, and let $x_0$ be the first $x_0 < \sqrt{2n}$ with this property. It follows from the previous description of the shape of quarter spheres near the sphere that $Q[x_0]$ is convex in the first quadrant and intersects the $r$-axis perpendicularly. Such a quarter sphere corresponds to a (closed) convex self-shrinker that is not the sphere, which contradicts Huisken's theorem for mean-convex self-shrinkers.

Next, we consider the case of outer-quarter spheres. As in the previous case, suppose to the contrary that some outer-quarter sphere intersects the $r$-axis with non-positive slope, and let $x_0$ be the first $x_0 > \sqrt{2n}$ with this property. It again follows from the previous description of the shape of quarter spheres near the sphere that $Q[x_0]$ intersects the $r$-axis perpendicularly; however, $Q[x_0]$ may not be convex in the first quadrant. We claim that the self-shrinker corresponding to $Q[x_0]$ is mean-convex. Writing $Q[x_0]$ as the graph $(x,u(x))$, where $u$ is a solution to~(\ref{x_graph_SSEq}), it is sufficient to show that $\Psi(x) =xu'-u$ does not vanish for $0 \leq x < x_0$. Since $\Psi(0)<0$ and $\Psi(x_0) = - \infty$, we need to show that $\Psi < 0$. If $\Psi$ has a non-negative maximum at some point $x_1 \in (0,x_0)$, then $$0 = \Psi'(x_1) = x_1 u''(x_1) = x_1 \left( 1 + u'(x_1)^2 \right)\left[ \frac{\Psi(x_1)}{2}+ \frac{n-1}{u(x_1)} \right] > 0,$$ which is a contradiction. Therefore, $\Psi <0$ and self-shrinker corresponding to $Q[x_0]$ is mean-convex, which contradicts Huisken's theorem for mean-convex self-shrinkers.

We conclude that the first graphical component of an outer-quarter sphere $Q[x_0]$ with $x_0 > \sqrt{2n}$ intersects the $r$-axis with positive slope, and the first graphical component of an inner-quarter sphere $Q[x_0]$ with $0<x_0<\sqrt{2n}$ intersects the $r$-axis with negative slope.
\end{proof}

%%%%%%%%%%%%%%%%%%%%
%%%%%%%%%%%%%%%%%%%%

\section*{Appendix C: Pictures of geodesics}

Here are some pictures of geodesics that correspond to immersed self-shrinkers.

\begin{figure}[p]
\label{fig:sphere}
\begin{center}
\includegraphics[width=.85 \textwidth]{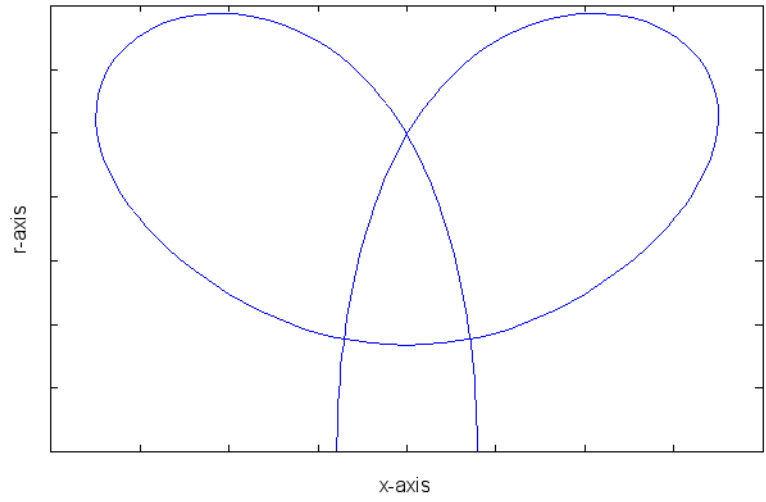}
\end{center}
\caption{A geodesic whose rotation about the $x$-axis is an immersed sphere self-shrinker.}
\end{figure}

\begin{figure}[p]
\label{fig:plane}
\begin{center}
\includegraphics[width=.85 \textwidth]{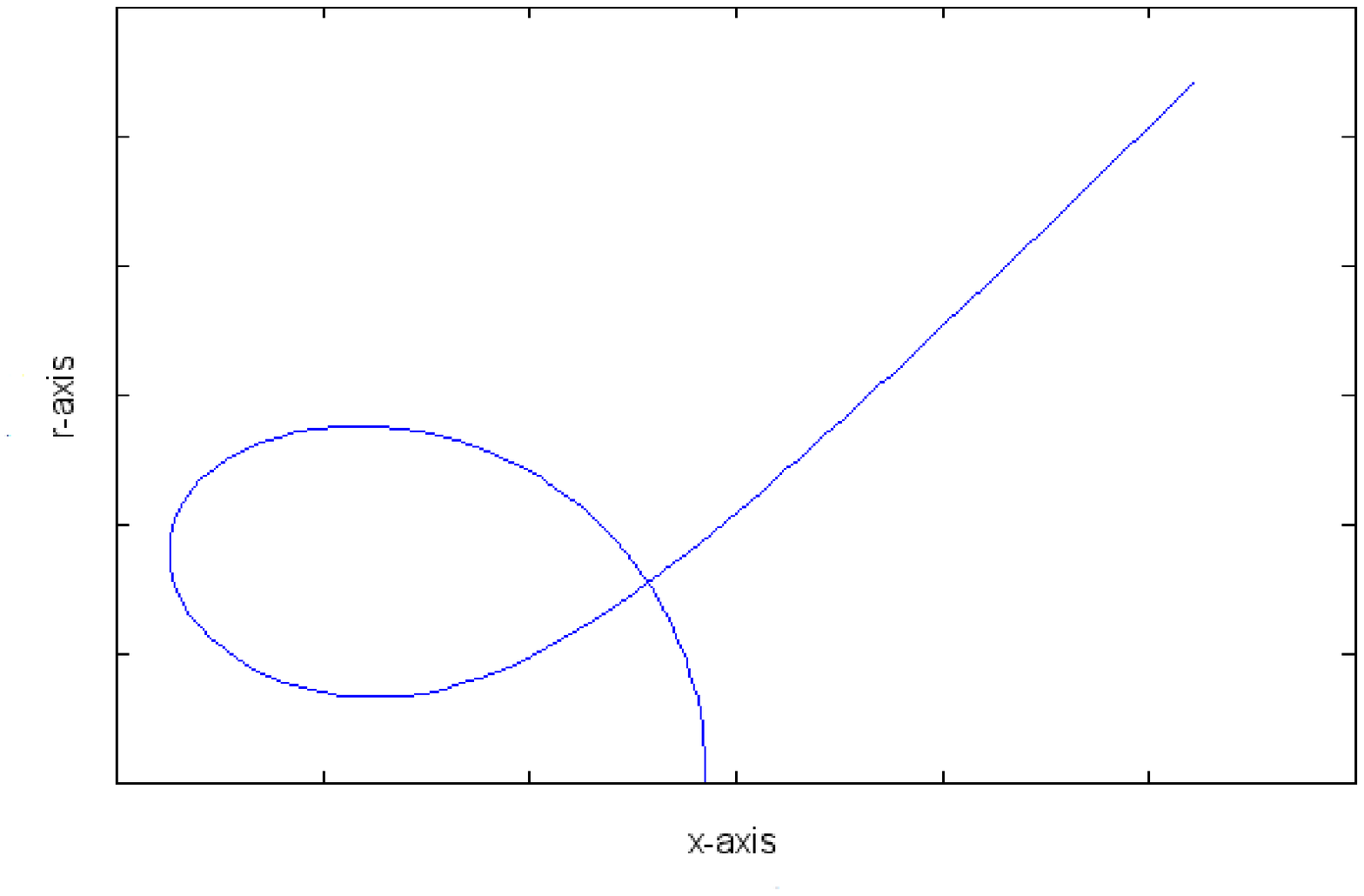}
\end{center}
\caption{A geodesic whose rotation about the $x$-axis is an immersed plane self-shrinker.}
\end{figure}

\begin{figure}[p]
\label{fig:cylinder}
\begin{center}
\includegraphics[width=.85 \textwidth]{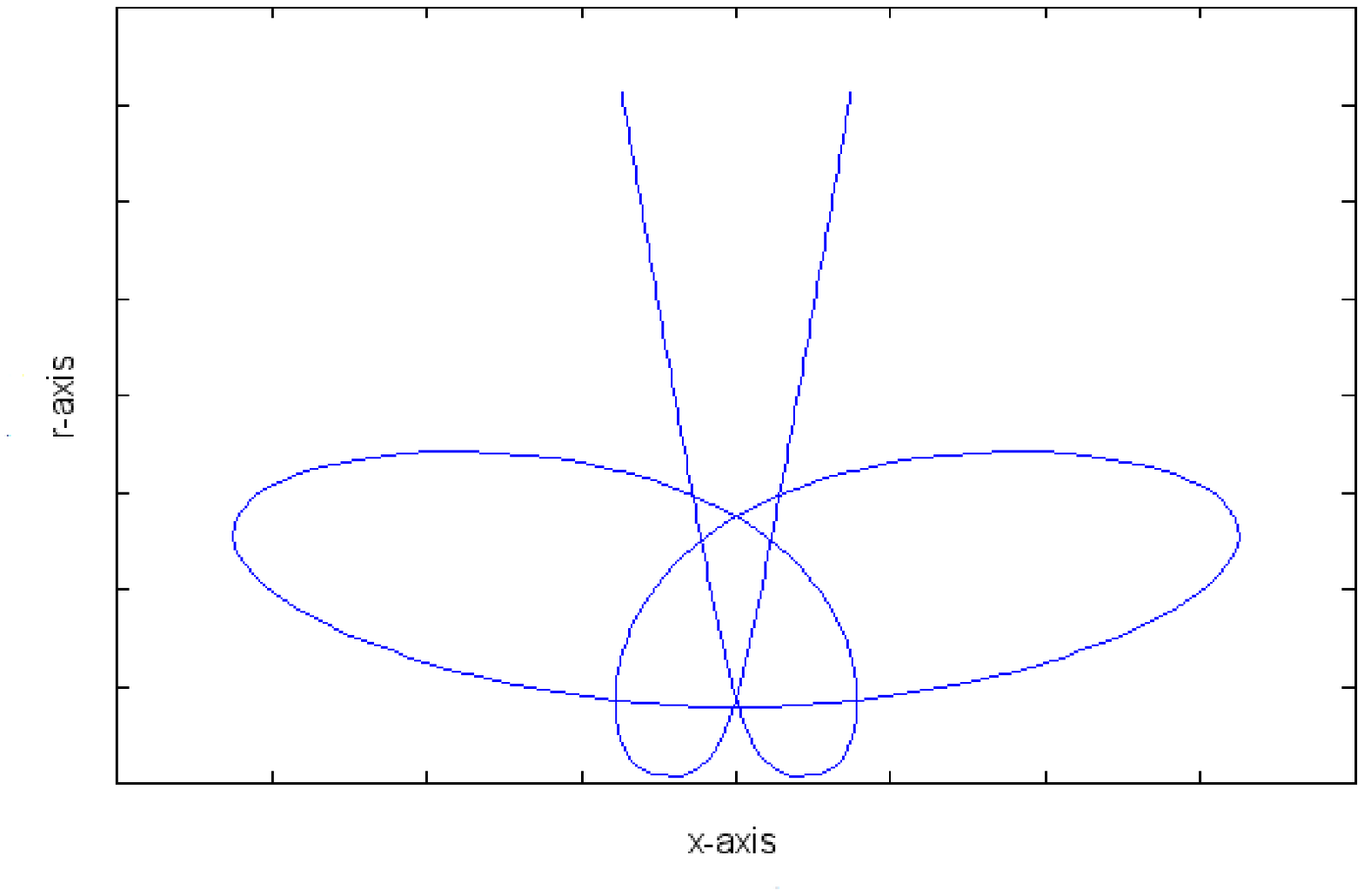}
\end{center}
\caption{A geodesic whose rotation about the $x$-axis is an immersed cylinder self-shrinker.}
\end{figure}

\begin{figure}[p]
\label{fig:torus}
\begin{center}
\includegraphics[width=.85 \textwidth]{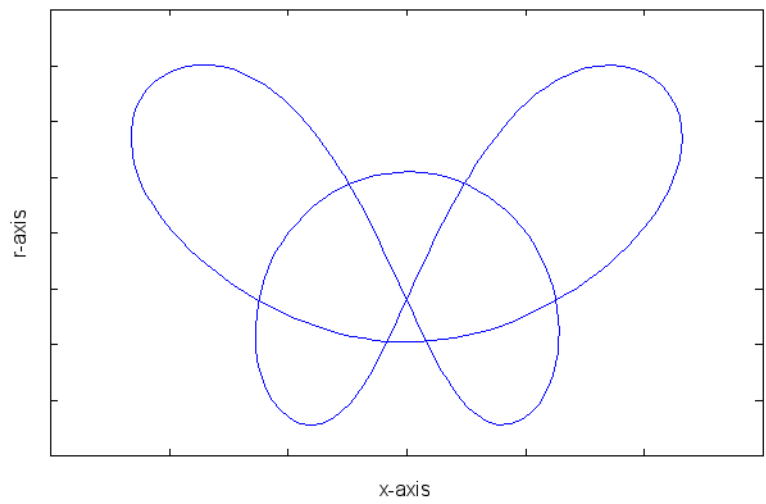}
\end{center}
\caption{A geodesic whose rotation about the $x$-axis is an immersed torus self-shrinker.}
\end{figure}

\clearpage

%%%%%%%%%%%%%%%%%%%%%%%%%%%%%%%%%%%%%%%
%%%%%%%%%%%%%%%%%%%%%%%%%%%%%%%%%%%%%%%

\bibliographystyle{amsplain}

\end{document}